\documentclass[11pt]{amsart}

\usepackage[T1]{fontenc}
\usepackage{lmodern}
\usepackage[utf8]{inputenc}
\usepackage{amsmath,amssymb,amsthm,mathtools}
\usepackage[margin=1.1in]{geometry}
\usepackage{enumitem}
\usepackage[hidelinks,
  pdftitle={Carrier ideals, tail obstructions, and remainder traces for ladder-system spaces},
  pdfauthor={Xing-Yu Hu},
  pdfsubject={Ladder-system spaces, carrier ideals, tail obstructions, and remainder traces},
  pdfkeywords={Delta-space, ladder-system space, countable metacompactness, carrier ideal, tail obstruction, nonstationary ideal, Stone-Cech remainder trace}
]{hyperref}

\newtheorem{theorem}{Theorem}[section]
\newtheorem{lemma}[theorem]{Lemma}
\newtheorem{proposition}[theorem]{Proposition}
\newtheorem{corollary}[theorem]{Corollary}
\newtheorem{problem}[theorem]{Problem}
\theoremstyle{definition}
\newtheorem{definition}[theorem]{Definition}

\newcommand{\omegaone}{\omega_1}
\newcommand{\cf}{\operatorname{cf}}
\newcommand{\cl}{\operatorname{cl}}
\newcommand{\supp}{\operatorname{supp}}
\newcommand{\acc}{\operatorname{acc}}
\newcommand{\Tr}{\operatorname{Tr}}
\newcommand{\cC}{\mathcal{C}}
\newcommand{\Jfin}{\mathcal{J}_{\rm fin}}
\newcommand{\GL}{G_L}
\DeclareRobustCommand{\affiliationline}{\normalfont\small \textsuperscript{a}\,School of Mathematics and Statistics, Hanjiang Normal University}
\DeclareRobustCommand{\addressline}{\normalfont\small No.~18 Beijing South Road, Shiyan 442000, Hubei Province, P.~R. China}
\DeclareRobustCommand{\correspondingline}{\normalfont\small \textsuperscript{*}Corresponding author: \texttt{huxingyu@hjnu.edu.cn}}

\title[Carrier ideals, tail obstructions, and traces]{Carrier ideals, tail obstructions, and remainder traces for ladder-system spaces}
\author[Xing-Yu Hu]{Xing-Yu Hu\textsuperscript{\lowercase{A},*}\\[4pt]
\affiliationline\\
\addressline\\
\correspondingline}


\begin{document}

\begin{abstract}
For a ladder-system space $X_L$ with carrier
$S\subseteq E^{\omegaone}_\omega$, the finite-label uniformization property
$M_{<\omega}$ characterizes countable metacompactness, and countable
metacompactness is equivalent to the $\Delta$-property.  Both equivalences are known for stationary carriers.  For arbitrary carriers,
an active-tail formulation gives a direct proof that $M_{<\omega}$ is
equivalent to the $\Delta$-property and leads to a support-finite decomposition
theorem, together with club-smallness and trace criteria that avoid explicit
ladder-position thresholds.

A club-gap argument, combined with Fodor's lemma, shows that finite and
countable tail multiplicity determine the same carrier ideal, namely
$\mathrm{NS}\restriction S$.  Subsets of the isolated part that meet each ladder in only finitely many points have
clopen remainder traces, and these traces form a generalized Boolean algebra.
All such traces are disjoint from the carrier part of the remainder.  Finally, the subcarriers whose
restricted spaces are $\sigma$-closed discrete form an ideal $\cC_L$ containing
$\mathrm{NS}\restriction S$.  If $X_L$ is a $\Delta$-space, a threshold-based
gluing argument shows that $\cC_L$ is a $\sigma$-ideal.  Whether this holds for
every ladder system remains open.
\end{abstract}

\maketitle

\section{Introduction}\label{sec:intro}

A \emph{ladder system} on a set $S$ of countable limit ordinals assigns to each
$\alpha\in S$ a strictly increasing sequence
$L_\alpha=\{\lambda_\alpha(n):n<\omega\}$ that is cofinal in $\alpha$ and
disjoint from $S$.  The associated ladder-system space $X_L$ has underlying
set $\omegaone$.  Every point of $D=\omegaone\setminus S$ is isolated, while a
basic neighborhood of $\alpha\in S$ consists of $\alpha$ together with a
cofinite tail of $L_\alpha$.  Any two distinct ordinal ladders meet in only finitely many points, so
$\{L_\alpha:\alpha\in S\}$ is an almost disjoint family on $D$.  Thus $X_L$ is
a locally compact, zero-dimensional generalized Isbell--Mr\'owka
$\Psi$-space of the type studied in \cite{Mrowka,Hrusak,HHH}.

A Tychonoff space $X$ is a $\Delta$-space if every decreasing sequence
$(A_n)_{n<\omega}$ of subsets with empty intersection can be expanded to a
decreasing sequence $(U_n)_{n<\omega}$ of open sets with
$A_n\subseteq U_n$ for all $n$ and $\bigcap_n U_n=\emptyset$.  K\k{a}kol and Leiderman proved that a Tychonoff space has this property
precisely when $C_p(X)$ is distinguished \cite[Theorem~2.1]{KLchar}.  Arkhangel'skii's
monograph \cite{Arhangelskii} is a standard source for the $C_p$-space
background.  For spaces of the form
$\Psi(D,\mathcal A)$, Leiderman and Szeptycki proved that the $\Delta$-property
is equivalent to countable metacompactness \cite[Prop.~4.1]{LS}.  Applying this theorem to the almost disjoint family of ordinal ladders shows
that the associated ladder-system space is a $\Delta$-space exactly when it is
countably metacompact.  Balogh, Eisworth, Gruenhage, Pavlov, and Szeptycki characterized countable metacompactness of $X_L$ by the
weak uniformization property $M_{<\omega}$ in the stationary-carrier setting
\cite[Claim~1]{BEGPS}.  Carvalho, Inamdar, and Rinot state the corresponding characterization for
ladder systems over stationary subsets of regular uncountable cardinals in
\cite[Fact~4.2]{CIR}, based on \cite[Prop.~4.1]{LS} and a direct
generalization of \cite[Claim~1]{BEGPS}.
Proposition~\ref{prop:active-normal-form} and
Lemma~\ref{lem:Mlessomega-equiv} give, in the present notation, an independent
proof of the equivalence
\[
        X_L\in\Delta \quad\Longleftrightarrow\quad L\text{ satisfies }M_{<\omega}.
\]
Recent work of Carvalho, Inamdar, and Rinot shows that countable
metacompactness of ladder-system spaces exhibits genuinely set-theoretic
behavior at higher cardinals \cite{CIR}.  For the classical $\Delta$-set terminology, see
Knight \cite{Knight}.  Memarpanahi and Szeptycki study the existence of
Lindel\"of $Q$-set spaces and Lindel\"of $\Delta$-set spaces.  They prove
that Moore's ZFC $L$-space is not a $Q$-set space and that, if all Aronszajn
trees are special, it is not a $\Delta$-set space \cite{MSQL}.  Further
structural results on
$\Delta$-spaces, including compactness and cardinality theorems, appear in
\cite{JvMSS}.

For general topological terminology, see Engelking \cite{Engelking}.  Club and
stationary subsets of $\omegaone$ are understood in the standard sense, and
Fodor's pressing-down lemma is used in its usual form.  For set-theoretic
notation and background in set-theoretic topology, see Jech \cite{Jech}, Kunen
\cite{Kunen}, and Kunen--Vaughan \cite{KunenVaughan}.

For arbitrary carriers, Lemma~\ref{lem:Mlessomega-equiv} recasts
$M_{<\omega}$ as an active-tail diagonalization property for decreasing sequences
of active carrier sets, replacing finite label sets by integer position thresholds.
This position formulation underlies the decomposition theorem, the hierarchy
of tail conditions, and the comparison of coarse support-fiber and club-block
hypotheses with explicit position bounds.

Support fibers encode the incidence relation between isolated points and
ladders.  Their cardinalities and distribution across club blocks lead to the
support-finite decomposition theorem and to the club-smallness and trace criteria.
The trace criterion, together with Fodor's lemma, shows that club-smallness
is equivalent to nonstationarity of the trace-accumulation set and to
countability of every support fiber.  Since subsets of $\omegaone$ are
countable exactly when they are bounded, club-smallness is also equivalent to
boundedness of every support fiber.  These incidence conditions are coarser than the explicit ladder-position
thresholds in $M_{<\omega}$ and in the active normal form.  The club-small decomposition ideal
$\mathcal J_{\rm cb}$ is precisely the $\sigma$-ideal of carriers with
countable support fibers.  A ZFC example, obtained by adding a common first
point to each ladder in the construction of \cite[Example~5.2]{LS}, produces a
$\sigma$-closed-discrete $\Delta$-space in which one isolated point lies on
every ladder.  Hence club-smallness, membership in
$\mathcal J_{\rm cb}$, countability of all support fibers, and
nonstationarity of the trace-accumulation set are not necessary for the
$\Delta$-property.  The sharper position-sensitive criteria therefore use explicit threshold
inequalities.

For $A\subseteq S$, let $A\in\Jfin$ mean that finite initial segments
can be deleted so that every isolated point lies on only finitely many remaining
tails indexed by $A$.  Let $A\in\mathcal J^\ast$ mean that finite initial
segments can be deleted so that every isolated point lies on at most countably
many remaining tails indexed by $A$.
The club-gap argument on a nonstationary carrier and Fodor's lemma on a
stationary carrier give
\[
        \Jfin=\mathcal J^\ast=\mathrm{NS}\restriction S.
\]
Finite and countable tail multiplicity therefore coincide at the carrier
level.  Both hold exactly for the nonstationary subcarriers.

The ideal $\mathcal I_L$ of subsets of the isolated part that meet each
ladder in only finitely many points induces a generalized Boolean algebra of
clopen traces in
the Stone--\v{C}ech remainder.  The compactification arguments use only standard facts about the Stone--\v{C}ech compactification
from \cite{GJ,Walker}.  These traces are disjoint from the $S$-part of the
remainder, and deleting any member of $\mathcal I_L$ from $X_L$ preserves the
$\Delta$-property.  No remainder characterization of the $\Delta$-property is
claimed.

Finally, let $\cC_L$ be the family of subcarriers $A\subseteq S$ for which
$X_A$ is $\sigma$-closed discrete.  It is an ideal containing every
nonstationary subcarrier.  On stationary carriers, the universal threshold
scheme is the strict-threshold form of \cite[Fact~4.2(3)]{CIR}.  For arbitrary
carriers, the equivalence between this scheme and $X_L\in\Delta$ is proved
directly.  Applying a
single threshold instance to countably many local colorings gives
\[
        X_L\in\Delta\quad\Longrightarrow\quad
        \cC_L\text{ is a }\sigma\text{-ideal}.
\]
Problem~\ref{prob:cC-zfc} asks whether this conclusion holds in ZFC without
the hypothesis $X_L\in\Delta$.

Throughout, $S\subseteq E^{\omegaone}_\omega=\{\alpha<\omegaone:\cf(\alpha)=\omega\}$
and $D=\omegaone\setminus S$.  When stationarity of the carrier is used, it is
stated explicitly.  For $\xi\in D$, write
$\supp_L(\xi)=\{\alpha\in S:\xi\in L_\alpha\}$ and let $p_\alpha(\xi)$ denote the
position of $\xi$ in $L_\alpha$.

\section{Preliminaries}\label{sec:preliminaries}

Under their standing convention that spaces are infinite and Tychonoff,
Leiderman and Szeptycki define $\Delta$-spaces by the condition in
\cite[Definition~1.2]{LS}.  For that class of spaces, their definition differs from the one below only notationally.  The definition below is stated for arbitrary topological spaces.

\begin{definition}
A topological space $X$ is a $\Delta$-space if for every decreasing sequence
\[
        A_0\supseteq A_1\supseteq A_2\supseteq\cdots
\]
of subsets of $X$ with $\bigcap_{n<\omega}A_n=\emptyset$, there is a decreasing sequence
\[
        U_0\supseteq U_1\supseteq U_2\supseteq\cdots
\]
of open subsets of $X$ such that $A_n\subseteq U_n$ for every $n<\omega$ and
\[
        \bigcap_{n<\omega}U_n=\emptyset.
\]
\end{definition}

\begin{lemma}\label{lem:decreasing}
In the definition of a $\Delta$-space, it is enough to find open sets $V_n$, not necessarily decreasing, such that $A_n\subseteq V_n$ for each $n<\omega$ and $\bigcap_{n<\omega}V_n=\emptyset$.  These sets may then be replaced by a decreasing sequence of open sets.
\end{lemma}

\begin{proof}
Set $U_n=\bigcap_{i\le n}V_i$.  Since $(A_n)$ is decreasing, $A_n\subseteq U_n$ for every $n$.  Moreover, $\bigcap_n U_n\subseteq\bigcap_n V_n=\emptyset$, so $\bigcap_n U_n=\emptyset$.
\end{proof}

Subspaces and topological sums of $\Delta$-spaces are again
$\Delta$-spaces.  In particular, every
discrete space is a $\Delta$-space.  Every countable $T_1$ space is also a
$\Delta$-space.

If $Y\subseteq X$ is a subspace, intersecting each witnessing open expansion in $X$ with $Y$ gives the required expansion in $Y$.  For
a topological sum, the defining expansion is applied component by component.  For the countable $T_1$ case, enumerate
$X=\{x_k:k<\omega\}$.  Given decreasing $(A_n)$ with empty intersection,
choose $n(k)$ with $x_k\notin A_{n(k)}$ and set
$U_n=X\setminus\{x_k:k\leq n\text{ and }n(k)\leq n\}$.  The removed set
is finite, so $U_n$ is open.  The sets $U_n$ decrease, contain $A_n$, and have
empty intersection.

\section{Main Results}\label{sec:main-results}

\subsection{Ladder-system spaces}\label{sec:spaces}

Unless stated otherwise, fix $S\subseteq E^{\omegaone}_\omega$, where
\[
        E^{\omegaone}_\omega=\{\alpha<\omegaone:\cf(\alpha)=\omega\}.
\]
Put $D=\omegaone\setminus S$.

Here a ladder system on $S$ means a family
\[
        L=\{L_\alpha:\alpha\in S\}
\]
such that each $L_\alpha=\{\lambda_\alpha(n):n<\omega\}$ is strictly increasing and cofinal in $\alpha$, and
\[
        L_\alpha\subseteq D
        \quad\text{for every }\alpha\in S.
\]
This requirement ensures that every ladder point is isolated in the associated
space.  Without it, ladder points may be non-isolated and $X_L$ need not be
locally compact.

For ladder systems on stationary carriers with $L_\alpha\subseteq D$ for every
$\alpha\in S$, the reduced presentation used here is homeomorphic to the
subspace of the standard doubled presentation in \cite[p.~130]{LS} obtained by deleting the isolated set
$S\times\{0\}$.  The homeomorphism sends $(\xi,0)$ to $\xi$ for
$\xi\in D$ and $(\alpha,1)$ to $\alpha$ for $\alpha\in S$.  The same reduced presentation is used for arbitrary carriers.

\begin{definition}
The ladder-system space $X_L$ has underlying set $\omegaone$.  Each point of $D$ is isolated.  For $\alpha\in S$ and $m<\omega$, put
\[
        N(\alpha,m)=\{\alpha\}\cup\{\lambda_\alpha(n):m\leq n<\omega\}.
\]
The sets $N(\alpha,m)$ form a local base at $\alpha$.
\end{definition}

For $\xi\in D$, define the ladder support of $\xi$ by
\[
        \supp_L(\xi)=\{\alpha\in S:\xi\in L_\alpha\}.
\]
If $\xi\in L_\alpha$, let $p_\alpha(\xi)$ be the unique integer $n$ such that $\xi=\lambda_\alpha(n)$.

If $\alpha<\beta$ are in $S$, then $L_\alpha\cap L_\beta$ is finite, since
$L_\beta\cap\alpha$ is finite.  Hence $X_L$ is a $\Psi(D,\mathcal A)$-space for
the almost disjoint family $\mathcal A=\{L_\alpha:\alpha\in S\}$, where
$L_\alpha$ corresponds to the carrier point $\alpha$.  The $\Psi$-space criterion of Leiderman and Szeptycki applies here \cite[Prop.~4.1]{LS}.
Compared with the doubled ladder-space presentation in
\cite[p.~189]{BEGPS}, this reduced presentation omits the separate isolated set indexed by $S$.  The omitted part is clopen and discrete, so this
convention does not change the $\Delta$-property or countable
metacompactness.

For $A\subseteq S$, write $L\restriction A=\langle L_\alpha:\alpha\in A\rangle$ for
the \emph{restricted ladder system} on the carrier $A$, and
\[
        X_A=A\cup\bigcup_{\alpha\in A}L_\alpha
\]
for its restricted ladder-system space, regarded as a subspace of $X_L$.  Its
isolated part is
\[
        D_A=\bigcup_{\alpha\in A}L_\alpha.
\]
The members of $A$ are the non-isolated points of $X_A$.  For each $\alpha\in A$, the sets
$N(\alpha,m)$, $m<\omega$, form a local base at $\alpha$.  By convention, $X_{L\restriction A}$ denotes the restricted space just defined, so
$X_A=X_{L\restriction A}$.  In particular, symbols such as $X_P$ and
$X_{c^{-1}\{i\}}$ always denote generated restricted spaces of this form.
For $\xi\in D_A$, the support in the restricted system is
\[
        \supp_{L\restriction A}(\xi)=A\cap\supp_L(\xi).
\]
The position $p_\alpha(\xi)$ is unchanged under restriction.  Lemma~\ref{lem:sigma-cd-coloring} and
Propositions~\ref{prop:tail-equivalence} and~\ref{prop:active-normal-form}
remain valid after replacing $S,D,L$ by $A,D_A,L\restriction A$.  Subsequent
applications to a subcarrier use these restricted forms.  When convenient, a
function on $D_A$ may be extended arbitrarily to $D$.

The space $X_L$ is Hausdorff, zero-dimensional, and locally compact.  For every
$\alpha\in S$ and $m<\omega$, the set $N(\alpha,m)$ is compact and open.  Indeed,
points of $D$ are isolated, and distinct ladders have finite intersection, so
distinct points can be separated by open sets.  Each set $N(\alpha,m)$ consists of a convergent sequence together with
its limit and is therefore compact.  It is open by definition.  Its
complement is open because every other ladder has only finite intersection with
the tail of $L_\alpha$.  Thus the local base at each non-isolated point consists of compact open sets.

In the terminology of Leiderman and Szeptycki, an infinite Tychonoff space is
$\sigma$-closed discrete when it is a countable union of closed discrete
subspaces.  For $X_L$, this terminology specializes as follows.

\begin{definition}[{\cite[p.~119]{LS}}]\label{def:sigma-cd}
$X_L$ is \emph{$\sigma$-closed discrete} if its underlying set is a countable union of closed discrete subspaces.
\end{definition}
When this holds, the ladder system $L$ is also called $\sigma$-closed discrete.

\begin{lemma}\label{lem:sigma-cd-coloring}
$X_L$ is $\sigma$-closed discrete if and only if there is $c:D\to\omega$ such that $c\restriction L_\alpha$ is finite-to-one for every $\alpha\in S$.
\end{lemma}

\begin{proof}
The set $S$ of non-isolated points is closed discrete, since
$N(\alpha,m)\cap S=\{\alpha\}$.  Given a $\sigma$-closed-discrete cover of
$X_L$, write
$X_L=\bigcup_n E_n$, with each $E_n$ closed discrete in $X_L$, and put
$A_n=E_n\cap D$.  If $A_n\cap L_\alpha$ were infinite, then $\alpha\in\cl_{X_L}E_n$.
Closedness would give $\alpha\in E_n$, but the infinitude of
$A_n\cap L_\alpha$ would contradict discreteness of $E_n$ at $\alpha$.  Hence
$A_n\cap L_\alpha$ is finite for every $\alpha$.  Thus $D$ is covered by
countably many sets, each of which meets every ladder in a finite set.

Conversely, if $D=\bigcup_n A_n$ and each $A_n$ meets every ladder in finitely many points, then
each $A_n$ is closed and discrete in $X_L$.  Indeed, if
$A_n\cap L_\alpha$ were infinite, then $\alpha$ would lie in the closure of
$A_n$.  Together with the closed discrete set $S$, the sets $A_n$ form a $\sigma$-closed-discrete cover of $X_L$.
Given such a decomposition, define $c(\xi)=\min\{n:\xi\in A_n\}$.  Then
$c^{-1}(n)\subseteq A_n$, so $c\restriction L_\alpha$ is finite-to-one.
If, on the other hand, $c\restriction L_\alpha$ is finite-to-one for all
$\alpha$, then $A_n=c^{-1}(n)$ meets each $L_\alpha$ in finitely many points and
$D=\bigcup_n A_n$.
\end{proof}

\subsection{Tail equivalence and normalization}

Deleting finitely many initial points from each ladder leaves the topology
unchanged.

\begin{definition}
Let $L$ and $M$ be ladder systems on the same set $S\subseteq E^{\omegaone}_\omega$, with the same isolated part $D=\omegaone\setminus S$.  The systems $L$ and $M$ are tail equivalent if
\[
        L_\alpha\triangle M_\alpha
\]
is finite for every $\alpha\in S$.
\end{definition}

\begin{proposition}\label{prop:tail-equivalence-homeomorphism}
If $L$ and $M$ are tail-equivalent ladder systems on the same $S$, then the identity map on $\omegaone$ is a homeomorphism
\[
        X_L\longrightarrow X_M.
\]
In particular,
\[
        X_L\in\Delta \quad\Longleftrightarrow\quad X_M\in\Delta.
\]
\end{proposition}

\begin{proof}
All points of $D$ are isolated in both spaces, so it remains only to compare neighborhoods of points of $S$.  Fix $\alpha\in S$.  Write
\[
        L_\alpha=\{\lambda_\alpha(n):n<\omega\},
        \qquad
        M_\alpha=\{\mu_\alpha(n):n<\omega\}.
\]
Both enumerations are increasing.

Let $m<\omega$.  Since $M_\alpha\setminus L_\alpha$ is finite and
\[
        L_\alpha\setminus\{\lambda_\alpha(n):m\leq n<\omega\}
\]
is finite, there is $k<\omega$ such that
\[
        \{\mu_\alpha(n):k\leq n<\omega\}
        \subseteq
        \{\lambda_\alpha(n):m\leq n<\omega\}.
\]
Thus every $L$-basic neighborhood of $\alpha$ contains an $M$-basic neighborhood.  Interchanging $L$ and $M$ shows that every $M$-basic neighborhood contains an $L$-basic neighborhood.  Hence the two local bases at $\alpha$ generate the same neighborhoods.  Since $\alpha$ was arbitrary, the identity map is a homeomorphism.  Because the $\Delta$-property is topological, $X_L\in\Delta$ if and only if $X_M\in\Delta$.
\end{proof}

\begin{corollary}\label{cor:tail-thinning-same-topology}
Let $h:S\to\omega$.  For each $\alpha\in S$, put
\[
        L^h_\alpha=\{\lambda_\alpha(n):h(\alpha)\leq n<\omega\}.
\]
Then $L^h=\{L^h_\alpha:\alpha\in S\}$ is tail equivalent to $L$, and
\[
        X_{L^h}=X_L
\]
as topological spaces.
\end{corollary}

\begin{proof}
For each $\alpha\in S$, the symmetric difference $L_\alpha\triangle L^h_\alpha$ is contained in the first $h(\alpha)$ points of $L_\alpha$.  Proposition~\ref{prop:tail-equivalence-homeomorphism} applies.
\end{proof}

Although the topology is unchanged, the support hypergraph and the associated fiber
conditions may vary with the chosen ladder presentation, since changing finitely many
initial points can alter how many ladders contain a fixed isolated point.
Support-based criteria are therefore stated for a fixed presentation, or for a
specified thinning such as $L^h$.

\subsection{Tail diagonalization and the active normal form}\label{sec:tail-diag}

For $X_L$, the $\Delta$-property can be characterized in terms of tails
chosen from the active ladders.

\begin{definition}
Let $L$ be a ladder system on $S$, and put $D=\omegaone\setminus S$.  An evanescent ladder pair is a pair of decreasing sequences
\[
        B_0\supseteq B_1\supseteq B_2\supseteq\cdots,\qquad
        C_0\supseteq C_1\supseteq C_2\supseteq\cdots.
\]
Here $B_n\subseteq D$ and $C_n\subseteq S$ for every $n$, and
\[
        \bigcap_{n<\omega}(B_n\cup C_n)=\emptyset.
\]
A tail diagonalization for $(B_n,C_n)_{n<\omega}$ is a sequence of functions
\[
        f_n:C_n\longrightarrow\omega
\]
such that for every $\xi\in D$ there is $n<\omega$ satisfying
\[
        \xi\notin B_n
        \quad\text{and}\quad
        p_\alpha(\xi)<f_n(\alpha)
        \text{ whenever }\alpha\in C_n\cap\supp_L(\xi).
\]
\end{definition}

At stage $n$, $B_n$ consists of the isolated points retained at that stage, whereas $f_n$
specifies how far to trim each active ladder.  The diagonalization condition
ensures that every isolated point is absent from $B_n$ and from all selected
ladder tails at some stage.

\begin{proposition}\label{prop:tail-equivalence}
For a ladder system $L$ on $S$, $X_L\in\Delta$ if and only if every evanescent ladder pair admits a tail diagonalization.
\end{proposition}

\begin{proof}
Assume first that $X_L\in\Delta$.  Let $(B_n,C_n)_{n<\omega}$ be an evanescent ladder pair, and put $A_n=B_n\cup C_n$.  Choose open sets $U_n$ with $A_n\subseteq U_n$ and $\bigcap_n U_n=\emptyset$.  For every $\alpha\in C_n$, choose $f_n(\alpha)<\omega$ such that
\[
        N(\alpha,f_n(\alpha))\subseteq U_n.
\]
Fix $\xi\in D$.  Since $\xi\notin\bigcap_n U_n$, there is $n$ with $\xi\notin U_n$.  Then $\xi\notin B_n$.  If $\alpha\in C_n\cap\supp_L(\xi)$ and $p_\alpha(\xi)\geq f_n(\alpha)$, then $\xi\in N(\alpha,f_n(\alpha))\subseteq U_n$, a contradiction.  Thus $(f_n)$ is a tail diagonalization.

Conversely, suppose that every evanescent ladder pair admits a tail diagonalization.  Let $(A_n)_{n<\omega}$ be a decreasing sequence of subsets of $X_L$ with empty intersection.  Put
\[
        B_n=A_n\cap D,
        \qquad
        C_n=A_n\cap S.
\]
Choose a tail diagonalization $f_n:C_n\to\omega$.  Define
\[
        V_n=B_n\cup C_n\cup
        \bigcup_{\alpha\in C_n}\{\lambda_\alpha(k):f_n(\alpha)\leq k<\omega\}.
\]
The set $V_n$ is open in $X_L$.  If $\alpha\in V_n\cap S$, then
$\alpha\in C_n$ and $V_n$ contains a tail of $L_\alpha$.  By construction,
$A_n\subseteq V_n$.

It remains to show that $\bigcap_n V_n=\emptyset$.  If $\alpha\in S$, then
$\alpha\notin C_n$ for some $n$, because $\bigcap_n A_n=\emptyset$.  Since
all ladder points lie in $D$, it follows that $\alpha\notin V_n$.  If $\xi\in D$, take
$n$ from the tail diagonalization.  Then $\xi\notin B_n$.  Moreover, every
$\alpha\in C_n\cap\supp_L(\xi)$ satisfies $p_\alpha(\xi)<f_n(\alpha)$.
Hence $\xi$ is not contained in any tail inserted at stage $n$, and so
$\xi\notin V_n$.  Thus the intersection of the $V_n$ is empty.
Lemma~\ref{lem:decreasing} completes the proof.
\end{proof}

Proposition~\ref{prop:tail-equivalence} reformulates the
open-set definition in terms of isolated points and active ladder tails.  The subsequent position-based sufficient conditions depend on the supports
$\supp_L(\xi)$ and the positions $p_\alpha(\xi)$.

The next proposition reduces the criterion to decreasing sequences of active carrier sets.

\begin{definition}
An active ladder sequence is a decreasing sequence
\[
        C_0\supseteq C_1\supseteq C_2\supseteq\cdots
\]
of subsets of $S$ with $\bigcap_{n<\omega}C_n=\emptyset$.  It is actively diagonalizable if there are functions
\[
        f_n:C_n\longrightarrow\omega
\]
such that for every $\xi\in D$ there is $n<\omega$ satisfying
\[
        p_\alpha(\xi)<f_n(\alpha)
        \quad\text{whenever }\alpha\in C_n\cap\supp_L(\xi).
\]
It is eventually actively diagonalizable if the functions may be chosen so that, for every $\xi\in D$, there is $n_\xi<\omega$ such that $p_\alpha(\xi)<f_n(\alpha)$ whenever $n\ge n_\xi$ and $\alpha\in C_n\cap\supp_L(\xi)$.
\end{definition}

\begin{proposition}\label{prop:active-normal-form}
For a ladder system $L$, the following are equivalent.
\begin{enumerate}[label=(\roman*),leftmargin=2em]
\item $X_L\in\Delta$.
\item Every evanescent ladder pair admits a tail diagonalization.
\item Every active ladder sequence is actively diagonalizable.
\item Every active ladder sequence is eventually actively diagonalizable.
\end{enumerate}
\end{proposition}

\begin{proof}
The equivalence of (i) and (ii) is Proposition~\ref{prop:tail-equivalence}.  The implication (ii)$\Rightarrow$(iii) follows by applying (ii) to the evanescent pair $(\emptyset,C_n)_{n<\omega}$.

Assume (iii), and let $(C_n)_{n<\omega}$ be an active ladder sequence.  Choose functions $f_n:C_n\to\omega$ witnessing active diagonalizability.  Define
\[
        g_n(\alpha)=\max\{f_i(\alpha): i\leq n\text{ and }\alpha\in C_i\},
        \qquad \alpha\in C_n.
\]
The maximum is well defined because the sequence $(C_n)$ is decreasing.  Fix $\xi\in D$.  There is $n_0$ such that
\[
        p_\alpha(\xi)<f_{n_0}(\alpha)
        \quad\text{for all }\alpha\in C_{n_0}\cap\supp_L(\xi).
\]
If $n\geq n_0$ and $\alpha\in C_n\cap\supp_L(\xi)$, then $\alpha\in C_{n_0}$ and
\[
        p_\alpha(\xi)<f_{n_0}(\alpha)\leq g_n(\alpha).
\]
Thus $(g_n)$ witnesses eventual active diagonalizability, proving (iii)$\Rightarrow$(iv).

Finally, assume (iv), and let $(B_n,C_n)_{n<\omega}$ be an evanescent ladder pair.  Apply (iv) to $(C_n)_{n<\omega}$.  Then there are functions $f_n:C_n\to\omega$ such that, for each $\xi\in D$, the active inequality holds for all sufficiently large $n$.  Since $(B_n)$ is decreasing and $\bigcap_n(B_n\cup C_n)=\emptyset$, each $\xi\in D$ also lies outside $B_n$ for all sufficiently large $n$.  For each $\xi$, choose a stage at which both requirements hold.  Then
\[
        \xi\notin B_n
        \quad\text{and}\quad
        p_\alpha(\xi)<f_n(\alpha)
        \quad\text{whenever }\alpha\in C_n\cap\supp_L(\xi).
\]
Thus $(f_n)$ is a tail diagonalization of the original evanescent pair.  By Proposition~\ref{prop:tail-equivalence}, $X_L\in\Delta$.
\end{proof}

By the contrapositive of Proposition~\ref{prop:active-normal-form}, if
$X_L\notin\Delta$, some active ladder sequence $(C_n)_{n<\omega}$ is not
actively diagonalizable.  Thus failure of the $\Delta$-property can be
witnessed by an active ladder sequence alone.

In this normal form, isolated points affect the condition only through their supports and
positions on active ladders, while the eventual requirement accounts for the
decreasing isolated-point sets $B_n$.

For stationary carriers, Balogh, Eisworth, Gruenhage, Pavlov, and Szeptycki
define $M_{<\omega}$ as follows \cite[p.~189]{BEGPS} and prove that it is
equivalent to countable metacompactness \cite[Claim~1]{BEGPS}.  Combined with
Lemma~\ref{lem:Mlessomega-equiv}, the active normal form expresses this
$\Delta$-criterion in terms of ladder positions.  In their formulation, the uniformizing map is defined on $\omegaone$.  Because every ladder is contained in $D$, a witness on $\omegaone$ may be restricted to $D$, while a witness on $D$ may be extended arbitrarily to $\omegaone$.  Either operation preserves the uniformization requirement.  For arbitrary
carriers, $M_{<\omega}$ is therefore formulated in this $D$-based form.

\begin{definition}\label{def:Mlessomega}
Let $L$ be a ladder system on $S\subseteq E^{\omegaone}_\omega$ with
$L_\alpha=\{\lambda_\alpha(n):n<\omega\}\subseteq D$.  The ladder system $L$ satisfies the \emph{countable-metacompactness uniformization property}
$M_{<\omega}$ if for every $f:S\to\omega$ there is
\[
        F:D\longrightarrow[\omega]^{<\omega}
\]
such that, for every $\alpha\in S$,
\[
        f(\alpha)\in F(\beta)
        \quad\text{for all but finitely many }\beta\in L_\alpha.
\]
\end{definition}

\begin{lemma}\label{lem:Mlessomega-equiv}
Let $L$ be a ladder system on $S\subseteq E^{\omegaone}_\omega$.  The
following are equivalent.
\begin{enumerate}[label=(\roman*),leftmargin=2em]
\item $L$ satisfies $M_{<\omega}$.
\item Every active ladder sequence is eventually actively diagonalizable
\textup{(}condition \textup{(iv)} of Proposition~\ref{prop:active-normal-form}\textup{)}.
\end{enumerate}
This equivalence is purely combinatorial and does not use the topology of
$X_L$.
\end{lemma}

\begin{proof}
Assume first that $L$ satisfies $M_{<\omega}$, and let
$C_0\supseteq C_1\supseteq C_2\supseteq\cdots$ be an active ladder sequence, so
$C_n\subseteq S$ and $\bigcap_n C_n=\emptyset$.  Since $(C_n)$ is decreasing
with empty intersection, each $\alpha\in C_0$ lies in a unique block
$C_{\rho(\alpha)}\setminus C_{\rho(\alpha)+1}$.  Define $f:S\to\omega$ by
$f(\alpha)=\rho(\alpha)$ for $\alpha\in C_0$ and $f(\alpha)=0$ for
$\alpha\in S\setminus C_0$.  Apply $M_{<\omega}$ to $f$ to obtain
$F:D\to[\omega]^{<\omega}$ with
\[
        f(\alpha)\in F(\beta)
        \quad\text{for all but finitely many }\beta\in L_\alpha\quad(\alpha\in S).
\]
For $\alpha\in S$ put
\[
        h(\alpha)=1+\max\Bigl(\{0\}\cup
        \{k<\omega:f(\alpha)\notin F(\lambda_\alpha(k))\}\Bigr).
\]
The exceptional set is finite because
$f(\alpha)\in F(\lambda_\alpha(k))$ for all but finitely many $k$, so the
maximum is well defined.  For each $n<\omega$, define $f_n:C_n\to\omega$ by
$f_n(\alpha)=h(\alpha)$ for $\alpha\in C_n$.  For each $\alpha$, the threshold $h(\alpha)$ is independent of the stage $n$.

To verify eventual active diagonalizability, fix $\xi\in D$ and set
$n_\xi=1+\max F(\xi)$ \textup{(}with $\max\emptyset=0$\textup{)}.  Let
$n\ge n_\xi$ and $\alpha\in C_n\cap\supp_L(\xi)$, and suppose toward a
contradiction that $p_\alpha(\xi)\ge f_n(\alpha)=h(\alpha)$.  By the
definition of $h(\alpha)$, every position $k\ge h(\alpha)$ satisfies
$f(\alpha)\in F(\lambda_\alpha(k))$.  Since $\xi=\lambda_\alpha(p_\alpha(\xi))$
and $p_\alpha(\xi)\ge h(\alpha)$, it follows that
\[
        f(\alpha)\in F(\xi).
\]
But $\alpha\in C_n$ implies $f(\alpha)=\rho(\alpha)\ge n\ge n_\xi>\max F(\xi)$,
so $f(\alpha)\notin F(\xi)$, a contradiction.  Hence for all $n\ge n_\xi$
and all $\alpha\in C_n\cap\supp_L(\xi)$, the inequality $p_\alpha(\xi)<f_n(\alpha)$ holds.
Thus $(f_n)_{n<\omega}$ witnesses eventual active diagonalizability of
$(C_n)$.

Conversely, assume every active ladder sequence
is eventually actively diagonalizable, and let $f:S\to\omega$ be given.  Put
\[
        C_n=\{\alpha\in S:f(\alpha)\ge n\},\qquad n<\omega.
\]
Then $C_0=S\supseteq C_1\supseteq C_2\supseteq\cdots$ is decreasing, and
$\bigcap_n C_n=\emptyset$ because every value of $f$ is finite.
Hence $(C_n)$ is an active ladder sequence.  Choose $f_n:C_n\to\omega$
witnessing eventual active diagonalizability.  Define $F:D\to[\omega]^{<\omega}$
by
\[
        F(\xi)=\bigl\{\,n<\omega:\exists\,\alpha\in\supp_L(\xi)
        \text{ with } f(\alpha)=n \text{ and } p_\alpha(\xi)\ge f_n(\alpha)\,\bigr\}.
\]

Fix $\xi\in D$.  Choose $n_\xi$ so that, for every $n\ge n_\xi$ and every $\alpha\in C_n\cap\supp_L(\xi)$,
$p_\alpha(\xi)<f_n(\alpha)$.
Suppose $n\in F(\xi)$, and choose $\alpha$ witnessing this membership.  Then
$f(\alpha)=n$, so
$\alpha\in C_n$, and $\alpha\in C_n\cap\supp_L(\xi)$ with
$p_\alpha(\xi)\ge f_n(\alpha)$.  By the choice of $n_\xi$ this forces
$n<n_\xi$.  Hence $F(\xi)\subseteq\{0,1,\dots,n_\xi-1\}$ is finite.

It remains to see that $F$ uniformizes $f$.  Fix $\alpha\in S$ and put
$n_0=f(\alpha)$, so $\alpha\in C_{n_0}$.  For every $\beta=\lambda_\alpha(k)\in L_\alpha$ with
$k\ge f_{n_0}(\alpha)$, $f(\alpha)=n_0$ and
$p_\alpha(\beta)=k\ge f_{n_0}(\alpha)$.  Hence
$n_0=f(\alpha)\in F(\beta)$.  This holds for all but the first $f_{n_0}(\alpha)$ members of $L_\alpha$, so
$f(\alpha)\in F(\beta)$ for all but finitely many $\beta\in L_\alpha$.  Thus $F$ witnesses
$M_{<\omega}$ for $f$.
\end{proof}

Lemma~\ref{lem:Mlessomega-equiv} gives the correspondence
\[
\begin{array}{ccl}
 f:S\to\omega\quad\text{as a labeling of the carrier}
 &\longleftrightarrow&
 (C_n)_{n<\omega},\quad C_n=\{\alpha\in S:f(\alpha)\ge n\},\\[2mm]
 F(\xi)\in[\omega]^{<\omega}\quad\text{as a finite label set at }\xi
 &\longleftrightarrow&
 f_n(\alpha)\in\omega\quad\text{as a position threshold}.
\end{array}
\]
For stationary carriers, the equivalence
\[
        L\text{ satisfies }M_{<\omega}\quad\Longleftrightarrow\quad X_L\in\Delta
\]
follows from \cite[Claim~1]{BEGPS} and \cite[Prop.~4.1]{LS}.  Carvalho, Inamdar, and Rinot state the corresponding
version over stationary subsets of regular uncountable cardinals
\cite[Fact~4.2]{CIR}, drawing on \cite[Prop.~4.1]{LS} and a direct
generalization of \cite[Claim~1]{BEGPS}.  For arbitrary carriers, Lemma~\ref{lem:Mlessomega-equiv} and
Proposition~\ref{prop:active-normal-form} together establish this equivalence directly.  In
this translation, finite label sets $F(\xi)$ become integer thresholds
$f_n(\alpha)$, while the eventual requirement accounts for the isolated-point sets
$B_n$ in the evanescent pairs of Proposition~\ref{prop:tail-equivalence}.  The support-finite decomposition, the tail criteria, and the examples all use
this position formulation.  On stationary carriers, the
single-function threshold condition of Definition~\ref{def:U} is the
strict-threshold version of \cite[Fact~4.2(3)]{CIR}.  For arbitrary carriers,
it gives an equivalent position-level criterion without quantification over
active decreasing sequences.  The active normal form, by contrast, is tailored
to the decomposition arguments.

\subsection{Support-finite decompositions}\label{sec:support-finite}

Local active diagonalizations can be combined across a support-finite partition.  The support of each isolated point meets only finitely many pieces,
so the corresponding eventual bounds can be synchronized at a single stage.
Empty pieces are ignored.

\begin{definition}
Let $\mathcal P$ be a partition of $S$.  The partition $\mathcal P$ is support-finite for $L$ if, for every $\xi\in D$, the set
\[
        \{P\in\mathcal P:\supp_L(\xi)\cap P\neq\emptyset\}
\]
is finite.
For $P\in\mathcal P$, write $L\restriction P$ for the restricted ladder system
\[
        \{L_\alpha:\alpha\in P\}
\]
on the carrier $P$.  Its associated subspace is
\[
        X_P=P\cup\bigcup_{\alpha\in P}L_\alpha.
\]
\end{definition}

\begin{theorem}\label{thm:support-finite-decomposition}
Let $L$ be a ladder system on $S$.  Suppose that $S$ has a support-finite partition $\mathcal P$ such that
\[
        X_P\in\Delta
        \quad\text{for every }P\in\mathcal P.
\]
Then $X_L\in\Delta$.
\end{theorem}

\begin{proof}
By Proposition~\ref{prop:active-normal-form}, it is enough to diagonalize every active ladder sequence on $S$.  Let
\[
        C_0\supseteq C_1\supseteq C_2\supseteq\cdots
\]
be an active ladder sequence on $S$.  Fix $P\in\mathcal P$.  Then
\[
        C_n^P=C_n\cap P,
        \qquad n<\omega.
\]
These sets form an active ladder sequence for the restricted ladder system on
$P$.  Since $X_P\in\Delta$, Proposition~\ref{prop:active-normal-form},
applied to that restricted system, gives functions
\[
        f_n^P:C_n^P\longrightarrow\omega
\]
such that, for every $\xi\in D_P$, the active inequality holds for all sufficiently large $n$.

Since $\mathcal P$ is a partition of $S$, each $\alpha\in C_n$ belongs to a unique $P\in\mathcal P$.  Define $f_n:C_n\to\omega$ by
\[
        f_n(\alpha)=f_n^P(\alpha)
        \quad\text{whenever }\alpha\in C_n\cap P.
\]

Fix $\xi\in D$.  By support-finiteness, there are only finitely many $P\in\mathcal P$ such that $\supp_L(\xi)\cap P\neq\emptyset$.  For each such piece $P$, the point $\xi$ lies in $D_P$, and the sequence $(f_n^P)_{n<\omega}$ satisfies the active inequality at $\xi$ for all sufficiently large $n$.  Choose a common stage $n_0$ that works for all these pieces.  If $n\geq n_0$ and $\alpha\in C_n\cap\supp_L(\xi)$, then $\alpha$ belongs to one of the finitely many relevant pieces $P$, and hence
\[
        p_\alpha(\xi)<f_n^P(\alpha)=f_n(\alpha).
\]
Thus $(f_n)_{n<\omega}$ witnesses eventual active diagonalizability of $(C_n)_{n<\omega}$.  Proposition~\ref{prop:active-normal-form} gives $X_L\in\Delta$.
\end{proof}

In particular, a support-finite partition of $S$ into countable pieces implies
$X_L\in\Delta$, because every countable restricted ladder-system space is a
$\Delta$-space by the countable $T_1$ argument in
Section~\ref{sec:preliminaries}.

\subsection{Crossing obstructions and support colorings}\label{sec:support-coloring}

The contrapositive of the decomposition theorem gives a crossing obstruction.  A partition
$\mathcal P$ of $S$ is locally $\Delta$ for $L$ when $X_P\in\Delta$ for every
$P\in\mathcal P$.  For $\xi\in D$, put
\[
        \operatorname{cr}_{\mathcal P}(\xi)=
        \{P\in\mathcal P:\supp_L(\xi)\cap P\neq\emptyset\}.
\]
Thus $\operatorname{cr}_{\mathcal P}(\xi)$ is the set of pieces of $\mathcal P$ met by the
support of $\xi$.

\begin{corollary}\label{cor:infinite-crossing-obstruction}
Let $\mathcal P$ be a locally $\Delta$ partition of $S$.  If $X_L\notin\Delta$,
then there is $\xi\in D$ for which $\operatorname{cr}_{\mathcal P}(\xi)$
is infinite.
\end{corollary}

\begin{proof}
If $|\operatorname{cr}_{\mathcal P}(\xi)|<\omega$ for every $\xi\in D$, then
$\mathcal P$ is support-finite.  Since $\mathcal P$ is locally $\Delta$,
Theorem~\ref{thm:support-finite-decomposition} gives $X_L\in\Delta$, a
contradiction.
\end{proof}

Corollary~\ref{cor:infinite-crossing-obstruction} therefore implies that if
$S$ is partitioned into countable sets and $X_L\notin\Delta$, there is an
isolated point whose support meets infinitely many pieces.

Equivalently, the criterion may be stated in terms of colorings.  Let
\[
        \mathcal H_L=
        \{\supp_L(\xi):\xi\in D\text{ and }\supp_L(\xi)\neq\emptyset\}
\]
be the support hypergraph of $L$.  A coloring $c:S\to I$ is support-finite for
$L$ if
\[
        |c[\supp_L(\xi)]|<\omega
        \quad\text{for every }\xi\in D.
\]
It is locally $\Delta$ for $L$ if $X_{c^{-1}\{i\}}\in\Delta$ for every nonempty
fiber.

\begin{corollary}\label{cor:support-coloring-criterion}
Let $c:S\to I$ be a support-finite coloring which is locally $\Delta$ for $L$.
Then $X_L\in\Delta$.  Consequently, if $X_L\notin\Delta$ and $c$ is locally
$\Delta$ for $L$, then $c[\supp_L(\xi)]$ is infinite for some $\xi\in D$.
In particular, if $S$ has a coloring into countable fibers and every support
uses only finitely many colors, then $X_L\in\Delta$.
\end{corollary}

\begin{proof}
The nonempty fibers of $c$ form a partition of $S$.  Support-finiteness of the
coloring is exactly support-finiteness of this partition.  The first assertion
therefore follows from Theorem~\ref{thm:support-finite-decomposition}.  The
second assertion is Corollary~\ref{cor:infinite-crossing-obstruction} applied
to the fiber partition.  If the fibers are countable, then each fiber space is a
$\Delta$-space by the countable $T_1$ argument in Section~\ref{sec:preliminaries}.
\end{proof}

\subsection{Club-small sets and the trace criterion}\label{sec:club-small}

Given $A\subseteq S$, each club $C$ determines an interval partition $\mathcal P_C$ of the carrier and the corresponding block ideal $\mathcal B_C(A)$.  Allowing $C$ to vary gives the
notion of a club-small set.  The family $\mathcal J_{\rm cb}$ consists of those sets that admit countable
support-point-finite covers by club-small sets, while the trace criterion
characterizes club-smallness itself.  These notions depend only on incidence across club blocks, not on the
explicit position thresholds in the $\Delta$-criterion.
Throughout $S\subseteq E^{\omegaone}_\omega$ and
$D=\omegaone\setminus S$.

\subsubsection{A club partition of the carrier}
Let $C$ be a club subset of $\omegaone$, and write its increasing enumeration as
\[
        C=\{\gamma_\eta:\eta<\omegaone\}.
\]
The enumeration is continuous because $C$ is closed.  For $\eta<\omegaone$, put
\[
        I_\eta(C)=S\cap [\gamma_\eta,\gamma_{\eta+1}).
\]
The sets $I_\eta(C)$ form a partition of $S\setminus \gamma_0$.  By convention, the countable initial part $S\cap\gamma_0$ is added to $I_0(C)$.  With this convention,
\[
        \mathcal P_C=\{I_\eta(C):\eta<\omegaone\}
\]
is a partition of $S$ into countable pieces.

\subsubsection{Block ideals and club-small sets}
\begin{definition}
Let $A\subseteq S$ and let $\mathcal P$ be a partition of $S$.  Define
\[
        \mathcal B_{\mathcal P}(A)=
        \{B\subseteq A: |\operatorname{cr}_{\mathcal P}^{B}(\xi)|<\omega
        \text{ for every }\xi\in D\}.
\]
Here
\[
        \operatorname{cr}_{\mathcal P}^{B}(\xi)=
        \{P\in\mathcal P:\supp_L(\xi)\cap B\cap P\neq\emptyset\}.
\]
When $C\subseteq\omegaone$ is club, write
\[
        \mathcal B_C(A)=\mathcal B_{\mathcal P_C}(A).
\]
Finally put
\[
        \mathcal B_{\rm club}(A)=
        \{B\subseteq A: B\in\mathcal B_C(A)\text{ for some club }C\subseteq\omegaone\}.
\]
\end{definition}

\begin{lemma}\label{lem:block-ideal}
For every partition $\mathcal P$ of $S$ and every $A\subseteq S$, the family $\mathcal B_{\mathcal P}(A)$ is an ideal on $A$.
\end{lemma}

\begin{proof}
The defining condition is inherited by subsets, so the family is downward closed.  If $B_0,B_1\in\mathcal B_{\mathcal P}(A)$ and $\xi\in D$, then
\[
        \operatorname{cr}_{\mathcal P}^{B_0\cup B_1}(\xi)
        \subseteq
        \operatorname{cr}_{\mathcal P}^{B_0}(\xi)
        \cup
        \operatorname{cr}_{\mathcal P}^{B_1}(\xi).
\]
The right-hand side is finite, so $B_0\cup B_1\in\mathcal B_{\mathcal P}(A)$.
\end{proof}

\begin{lemma}\label{lem:coarsening}
Let $A\subseteq S$, and let $\mathcal Q$ be a coarsening of a partition
$\mathcal P$ of $S$.  If $B\in\mathcal B_{\mathcal P}(A)$, then
$B\in\mathcal B_{\mathcal Q}(A)$.
\end{lemma}

\begin{proof}
Every member of $\mathcal Q$ is a union of members of $\mathcal P$.  For each
$\xi\in D$, every $\mathcal Q$-piece meeting $\supp_L(\xi)\cap B$ therefore
contains a $\mathcal P$-piece meeting the same set.  Hence only finitely many
$\mathcal Q$-pieces meet $\supp_L(\xi)\cap B$.
\end{proof}

The family $\mathcal B_{\rm club}(A)$ is also an ideal on $A$.  A club witnessing $B\in\mathcal B_{\rm club}(A)$ also witnesses every subset of $B$, so the family is downward
closed.  If $B_0,B_1\in\mathcal B_{\rm club}(A)$, choose clubs $C_0,C_1$ such
that $B_i\in\mathcal B_{C_i}(A)$ for $i=0,1$, and put $E=C_0\cap C_1$.  Under the convention that the countable initial segment determined by each club
is included in its first block, the club partition $\mathcal P_E$ coarsens each
$\mathcal P_{C_i}$.  By Lemma~\ref{lem:coarsening}, both $B_i$ belong to
$\mathcal B_E(A)$, and Lemma~\ref{lem:block-ideal} gives
$B_0\cup B_1\in\mathcal B_E(A)\subseteq\mathcal B_{\rm club}(A)$.

\begin{definition}
A set $A\subseteq S$ is club-small if
\[
        A\in\mathcal B_{\rm club}(A).
\]
A sequence $(A_m)_{m<\omega}$ of subsets of $S$ is support-point-finite if, for every $\xi\in D$, the set
\[
        \{m<\omega:\supp_L(\xi)\cap A_m\neq\emptyset\}
\]
is finite.
\end{definition}

\begin{lemma}\label{lem:club-small-delta}
If $A\subseteq S$ is club-small, then the restricted ladder-system space $X_A$ is a $\Delta$-space.
\end{lemma}

\begin{proof}
Choose a club $C\subseteq\omegaone$ such that $A\in\mathcal B_C(A)$.  The sets
$A\cap I_\eta(C)$, $\eta<\omegaone$, form a partition of $A$ into
countable pieces.  For every $\xi\in D$, only finitely many of these pieces meet $\supp_L(\xi)\cap A$.  Hence this partition is support-finite for the restricted ladder system on $A$.  For each piece $P$, the associated space $X_P$ is countable and therefore a $\Delta$-space by the countable $T_1$ argument in Section~\ref{sec:preliminaries}.  Applying Theorem~\ref{thm:support-finite-decomposition} to the restricted ladder system on $A$ gives $X_A\in\Delta$.
\end{proof}

\subsubsection{The club-small decomposition family}
\begin{definition}
Assume $S\subseteq E^{\omegaone}_\omega$.  Define $\mathcal J_{\rm cb}$ to be the family of all sets $A\subseteq S$ for which there is a sequence $(A_m)_{m<\omega}$ of club-small subsets of $S$ such that
\[
        A\subseteq \bigcup_{m<\omega}A_m
\]
and the sequence $(A_m)_{m<\omega}$ is support-point-finite.  The family so defined is the club-small decomposition family associated with $L$.
\end{definition}

\begin{lemma}\label{lem:Jcb-ideal}
The family $\mathcal J_{\rm cb}$ is an ideal on $S$.  Moreover,
\[
        \mathcal B_{\rm club}(S)\subseteq \mathcal J_{\rm cb}.
\]
\end{lemma}

\begin{proof}
If $B\subseteq A$ and $A\in\mathcal J_{\rm cb}$, every support-point-finite cover witnessing $A\in\mathcal J_{\rm cb}$ also covers $B$.  Hence the family is downward closed.  Suppose $A,B\in\mathcal J_{\rm cb}$.  Choose support-point-finite covers
\[
        A\subseteq\bigcup_{m<\omega}A_m,
        \qquad
        B\subseteq\bigcup_{m<\omega}B_m.
\]
All $A_m$ and $B_m$ are club-small.  Enumerate the family
\[
        \{A_m:m<\omega\}\cup\{B_m:m<\omega\}
\]
as a countable sequence.  For each $\xi\in D$, the support of $\xi$ meets only finitely many of the
$A_m$ and only finitely many of the $B_m$.  It therefore meets only finitely
many sets in the combined sequence.  Thus $A\cup B\in\mathcal J_{\rm cb}$.

If $A\in\mathcal B_{\rm club}(S)$, set $A_0=A$ and $A_m=\emptyset$ for $m>0$.  This sequence is support-point-finite and consists of club-small sets, so $A\in\mathcal J_{\rm cb}$.
\end{proof}

\begin{proposition}\label{prop:Jcb-subspace-delta}
If $A\in\mathcal J_{\rm cb}$, then the restricted ladder-system space $X_A$ is a $\Delta$-space.
\end{proposition}

\begin{proof}
Choose a support-point-finite cover $A\subseteq\bigcup_{m<\omega}A_m$ by club-small sets.  Put
\[
        P_m=(A\cap A_m)\setminus\bigcup_{j<m}(A\cap A_j),
        \qquad m<\omega.
\]
Then $(P_m)_{m<\omega}$ partitions $A$.  Each $P_m$ is contained in the club-small set $A_m$, and club-smallness is inherited by subsets, so each $P_m$ is club-small.  By Lemma~\ref{lem:club-small-delta}, $X_{P_m}\in\Delta$ for every $m<\omega$.

The partition $(P_m)_{m<\omega}$ is support-finite for the restricted ladder system on $A$.  If $\supp_L(\xi)$ meets $P_m$, then it also meets $A_m$, which occurs for only finitely many $m$ by support-point-finiteness of the cover.  Theorem~\ref{thm:support-finite-decomposition}, applied to the restricted ladder system on $A$, gives $X_A\in\Delta$.
\end{proof}

In particular, applying Proposition~\ref{prop:Jcb-subspace-delta} with $A=S$
shows that $S\in\mathcal J_{\rm cb}$ implies $X_L\in\Delta$.

\begin{lemma}\label{lem:bounded-club-small}
Every bounded subset of $S\subseteq E^{\omegaone}_\omega$ is club-small.  Consequently, every bounded subset of $S$ belongs to $\mathcal J_{\rm cb}$.
\end{lemma}

\begin{proof}
Let $A\subseteq S$ be bounded, and choose $\gamma<\omegaone$ with $A\subseteq\gamma$.  Choose a club $C\subseteq\omegaone$ such that $0,\gamma\in C$ and $C\cap\gamma\subseteq\{0\}$.  Then all points of $A$ lie in a single interval of the club partition $\mathcal P_C$.  For every $\xi\in D$, the set of $C$-blocks meeting $\supp_L(\xi)\cap A$ has size at most one.  Thus $A\in\mathcal B_{\rm club}(A)$, so $A$ is club-small.  The same club $C$ shows that $A\in\mathcal B_C(S)$ and hence $A\in\mathcal B_{\rm club}(S)$.  Lemma~\ref{lem:Jcb-ideal} then gives $A\in\mathcal J_{\rm cb}$.
\end{proof}

\subsubsection{The trace-accumulation criterion}
For $B\subseteq\omegaone$, write
$\acc(B)=\{\delta<\omegaone:0<\delta=\sup(B\cap\delta)\}$.
For $A\subseteq S$, define the trace-accumulation set
\[
        \Tr_L(A)=\bigcup_{\xi\in D}\acc(W_A(\xi)),
        \qquad
        W_A(\xi)=\supp_L(\xi)\cap A.
\]
Thus $\delta\in\Tr_L(A)$ exactly when, for some isolated point $\xi$, the
points $\alpha\in A$ whose ladders contain $\xi$ are cofinal in $\delta$.

\begin{lemma}\label{lem:infinite-crossing-acc}
Let $C\subseteq\omegaone$ be club and let $W\subseteq S$.  Then $W$ meets infinitely many blocks of the club partition $\mathcal P_C$ if and only if
\[
        \acc(W)\cap\acc(C)\neq\emptyset.
\]
\end{lemma}

\begin{proof}
Enumerate $C$ increasingly and continuously as $C=\{\gamma_\eta:\eta<\omegaone\}$.  Suppose first that $W$ meets infinitely many blocks.  Choose strictly increasing indices $\eta_k$ and points
\[
        x_k\in W\cap[\gamma_{\eta_k},\gamma_{\eta_k+1})
        \qquad(k<\omega).
\]
Let $\lambda=\sup_k\eta_k$ and $\delta=\gamma_\lambda$.  Then $\lambda$ is a countable limit ordinal, so $\delta\in\acc(C)$.  Since $x_k\geq\gamma_{\eta_k}$ and $\sup_k\gamma_{\eta_k}=\delta$, the set $W\cap\delta$ is cofinal in $\delta$.  Hence $\delta\in\acc(W)$.

Conversely, suppose that $\delta\in\acc(W)\cap\acc(C)$.  Since $C\cap\delta$ is cofinal in $\delta$, there are infinitely many blocks of $\mathcal P_C$ below $\delta$.  Since $W\cap\delta$ is cofinal in $\delta$, it cannot be contained in a finite union of those blocks.  Hence $W$ meets infinitely many $C$-blocks.
\end{proof}

\begin{proposition}\label{prop:trace-club-small}
Assume $S\subseteq E^{\omegaone}_\omega$.  For every $A\subseteq S$,
\[
        A\text{ is club-small}
        \quad\Longleftrightarrow\quad
        \Tr_L(A)\text{ is nonstationary in }\omegaone.
\]
\end{proposition}

\begin{proof}
Suppose first that $A$ is club-small.  Choose a club $C\subseteq\omegaone$ such that for every $\xi\in D$, the support fiber $W_A(\xi)$ meets only finitely many blocks of $\mathcal P_C$.  By Lemma~\ref{lem:infinite-crossing-acc},
\[
        \acc(W_A(\xi))\cap\acc(C)=\emptyset
        \qquad(\xi\in D).
\]
Thus $\Tr_L(A)\cap\acc(C)=\emptyset$.  Since $\acc(C)$ is club in $\omegaone$, the set $\Tr_L(A)$ is nonstationary.

Conversely, suppose that $\Tr_L(A)$ is nonstationary.  Choose a club $C\subseteq\omegaone$ disjoint from $\Tr_L(A)$.  Since $\acc(C)\subseteq C$, it is also disjoint from $\Tr_L(A)$.  If some $W_A(\xi)$ met infinitely many blocks of $\mathcal P_C$, Lemma~\ref{lem:infinite-crossing-acc} would give
\[
        \acc(W_A(\xi))\cap\acc(C)\neq\emptyset.
\]
Any point in this intersection belongs to both $C$ and $\Tr_L(A)$, contradicting $C\cap\Tr_L(A)=\emptyset$.  Thus $C$ witnesses that $A$ is club-small.
\end{proof}

\subsection{A Fodor collapse for club-smallness, traces, and support fibers}\label{sec:fodor}

Club-smallness and the trace condition depend only on the incidence relation
$\xi\in L_\alpha$, not on the positions $p_\alpha(\xi)$.  Proposition~\ref{prop:trace-club-small}
and Fodor's lemma show that both are equivalent to countability of the fibers
$W_A(\xi)=\supp_L(\xi)\cap A$.  The common-first-point example of
Proposition~\ref{prop:common-first-point-example}, constructed in
Section~\ref{sec:elementary-positive}, shows that none of these conditions is
necessary for the $\Delta$-property.  By contrast, FTM and EPB
impose explicit inequalities on $p_\alpha(\xi)$.

\begin{lemma}\label{lem:fodor-collapse}
Assume $S\subseteq E^{\omegaone}_\omega$ and let $A\subseteq S$.  The following are
equivalent.
\begin{enumerate}[label=(\roman*),leftmargin=2em]
\item $\Tr_L(A)$ is stationary in $\omegaone$.
\item There is $\xi\in D$ with $|W_A(\xi)|=\aleph_1$.
\item $\Tr_L(A)$ contains a club.
\end{enumerate}
\end{lemma}

\begin{proof}
(ii)$\Rightarrow$(iii).  Let $\xi\in D$ with $W:=W_A(\xi)$ uncountable.  An
uncountable subset of $\omegaone$ is unbounded.  The set $\acc(W)$ is closed.
If $\delta_k\in\acc(W)$ with $\delta_k\uparrow\delta$, then $W\cap\delta$ is
cofinal in $\delta$, so $\delta\in\acc(W)$.  It is unbounded.  Given
$\beta<\omegaone$, choose $w_0\in W$ with $w_0>\beta$ and $w_{k+1}\in W$ with
$w_{k+1}>w_k$.  Then $\delta=\sup_k w_k\in\acc(W)$ and $\delta>\beta$.  Thus
$\acc(W)\subseteq\Tr_L(A)$ is club.

Every club is stationary, so (iii)$\Rightarrow$(i).

(i)$\Rightarrow$(ii).  Suppose $\Tr_L(A)$ is stationary.  Then
$T=\Tr_L(A)\cap\mathrm{Lim}(\omegaone)$ is stationary.  For $\gamma\in T$ choose
$\xi_\gamma\in D$ with $\gamma\in\acc(W_A(\xi_\gamma))$.  Then
$W_A(\xi_\gamma)\cap\gamma$ is cofinal in $\gamma$, so choose
$\alpha_\gamma\in W_A(\xi_\gamma)\cap\gamma$.  Since
$\alpha_\gamma\in\supp_L(\xi_\gamma)$ and $L_{\alpha_\gamma}\subseteq\alpha_\gamma$,
$\xi_\gamma<\alpha_\gamma<\gamma$.  Thus
$\gamma\mapsto\xi_\gamma$ is regressive.  By Fodor's lemma, there are a
stationary set $T^\ast\subseteq T$ and a fixed $\xi^\ast\in D$ such that
$\xi_\gamma=\xi^\ast$ for every $\gamma\in T^\ast$.  Hence $\acc(W_A(\xi^\ast))$ contains the stationary set $T^\ast$ and is
therefore stationary.  Moreover, $W_A(\xi^\ast)$ is unbounded and uncountable.
\end{proof}

\begin{corollary}\label{cor:trichotomy}
For $A\subseteq S\subseteq E^{\omegaone}_\omega$, the following are equivalent.
\textup{(i)} $A$ is club-small, \textup{(ii)} $\Tr_L(A)$ is nonstationary,
\textup{(iii)} $W_A(\xi)$ is countable for every $\xi\in D$, \textup{(iv)}
$W_A(\xi)$ is bounded for every $\xi\in D$.
\end{corollary}

\begin{proof}
(i)$\Leftrightarrow$(ii) is Proposition~\ref{prop:trace-club-small}.

Lemma~\ref{lem:fodor-collapse} says that $\Tr_L(A)$ is stationary if and only if some fiber $W_A(\xi)$ is uncountable.  Negating this statement gives (ii)$\Leftrightarrow$(iii).

Finally, (iii)$\Leftrightarrow$(iv) holds because a subset of $\omegaone$ is
countable if and only if bounded.
\end{proof}

\begin{corollary}\label{cor:sigma-ideal}
Assume $S\subseteq E^{\omegaone}_\omega$.  Then
\[
        \mathcal J_{\rm cb}=\mathcal B_{\rm club}(S)
        =\{A\subseteq S:|W_A(\xi)|\le\aleph_0\ \text{for every }\xi\in D\}.
\]
This family is closed under countable unions.
\end{corollary}

\begin{proof}
By Corollary~\ref{cor:trichotomy}, a set $A$ is club-small if and only if
every fiber $W_A(\xi)$ is countable.  Hence $\mathcal B_{\rm club}(S)$ is exactly the family of sets $A\subseteq S$ for which every fiber $W_A(\xi)$ is countable.  That family is closed under countable unions.  If each $A_m$ belongs to it, then
$W_{\bigcup_m A_m}(\xi)=\bigcup_m W_{A_m}(\xi)$ is countable for every $\xi\in D$.
For the first equality, the inclusion $\mathcal B_{\rm club}(S)\subseteq\mathcal J_{\rm cb}$ follows from
Lemma~\ref{lem:Jcb-ideal}.  Conversely, suppose that $A\in\mathcal J_{\rm cb}$, and let $(A_m)_{m<\omega}$ be a support-point-finite cover of $A$ by club-small sets.  Each $W_{A_m}(\xi)$ is countable, and
$W_A(\xi)\subseteq\bigcup_m W_{A_m}(\xi)$.  Hence every fiber $W_A(\xi)$ is countable,
so $A$ is club-small.
\end{proof}

Corollary~\ref{cor:sigma-ideal} shows that the support-point-finiteness clause
in the definition of $\mathcal J_{\rm cb}$ is redundant and that
$\mathcal J_{\rm cb}$ is a $\sigma$-ideal.  Corollary~\ref{cor:trichotomy}\,(iv)$\Rightarrow$(i)
shows that if every support
fiber $W_A(\xi)$ is bounded, then one club witnesses club-smallness
simultaneously for all $\xi\in D$.

\subsection{Position-based sufficient conditions and obstructions}\label{sec:elementary-positive}

Finite support controls incidence directly.  FTM and EPB generalize the position
bounds in the $(\text{limit})+n$ ladder systems of
\cite[Example~5.2]{LS}, replacing that fixed shape by carrier-dependent
thresholds and point-dependent finite offsets.

\begin{proposition}\label{prop:finite-support}
If $\supp_L(\xi)$ is finite for every $\xi\in D$, then $X_L$ is a $\Delta$-space.
\end{proposition}

\begin{proof}
Let $(B_n,C_n)_{n<\omega}$ be an evanescent ladder pair.  Set $f_n(\alpha)=0$ for every $n<\omega$ and every $\alpha\in C_n$.

Fix $\xi\in D$.  Since $(B_n,C_n)_{n<\omega}$ is evanescent, there is $n_0$ such that
$\xi\notin B_n$ for all $n\geq n_0$.  For each $\alpha\in\supp_L(\xi)$, there is $n_\alpha$ such that $\alpha\notin C_n$ for all $n\geq n_\alpha$.  Since $\supp_L(\xi)$ is finite, choose
\[
        n\geq n_0
        \quad\text{and}\quad
        n\geq n_\alpha\text{ for all }\alpha\in\supp_L(\xi).
\]
Then $\xi\notin B_n$ and $C_n\cap\supp_L(\xi)=\emptyset$.  Thus the functions $f_n\equiv0$ form a tail diagonalization.  Proposition~\ref{prop:tail-equivalence} gives $X_L\in\Delta$.

\end{proof}

Finite tail multiplicity is the analogous condition after tail thinning,
and its contrapositive gives a necessary condition for any counterexample.

\begin{definition}
A ladder system $L$ has finite tail multiplicity if there is a function
\[
        h:S\longrightarrow \omega
\]
such that for every $\xi\in D$ the set
\[
        \{\alpha\in \supp_L(\xi): p_\alpha(\xi)\geq h(\alpha)\}
\]
is finite.
After deleting the first $h(\alpha)$ points from each ladder $L_\alpha$, the condition says that every isolated point lies on only finitely many remaining ladders.
\end{definition}

\begin{proposition}\label{prop:finite-tail}
If $L$ has finite tail multiplicity, then $X_L$ is a $\Delta$-space.
\end{proposition}

\begin{proof}
Let $h:S\to\omega$ witness finite tail multiplicity.  Replace each $L_\alpha$ by the tail obtained after deleting its first $h(\alpha)$ points.  By Corollary~\ref{cor:tail-thinning-same-topology}, this does not change the associated topology.  In the resulting ladder system every isolated point lies on only finitely many ladders.  Proposition~\ref{prop:finite-support} applies.
\end{proof}

\begin{corollary}\label{cor:non-delta-tail}
If $X_L\notin\Delta$, then for every function $h:S\to\omega$ there is an isolated point $\xi\in D$ such that
\[
        \{\alpha\in \supp_L(\xi):p_\alpha(\xi)\geq h(\alpha)\}
\]
is infinite.
\end{corollary}

\begin{proof}
This is the contrapositive of Proposition~\ref{prop:finite-tail}.
\end{proof}

Finite tail multiplicity admits the following weakening.  An isolated point may remain on
infinitely many thinned ladders, provided its positions are eventually bounded
up to a finite offset depending on that point.

\begin{definition}
The ladder system $L$ has \emph{eventual position-bounded tails} (EPB) if there is a function
\[
        h:S\longrightarrow\omega
\]
such that for every $\xi\in D$ there is $q_\xi<\omega$ for which
\[
        \{\alpha\in\supp_L(\xi):p_\alpha(\xi)\geq h(\alpha)+q_\xi\}
\]
is finite.
\end{definition}

\begin{proposition}\label{prop:position-bounded}
If $L$ has eventual position-bounded tails, then $X_L$ is a $\Delta$-space.
\end{proposition}

\begin{proof}
Let $h:S\to\omega$ witness eventual position-boundedness, and let $(B_n,C_n)_{n<\omega}$ be an evanescent ladder pair.  For $\alpha\in C_n$ put
\[
        f_n(\alpha)=h(\alpha)+n+1.
\]
It remains to verify that $(f_n)_{n<\omega}$ is a tail diagonalization.

Fix $\xi\in D$.  Choose $q<\omega$ such that
\[
        F=\{\alpha\in\supp_L(\xi):p_\alpha(\xi)\geq h(\alpha)+q\}
\]
is finite.  Because the pair is evanescent and both $(B_n)$ and $(C_n)$ are decreasing, $\xi$ is eventually outside $B_n$, and every member of the finite set $F$ is eventually outside $C_n$.  Choose $n<\omega$ such that
$n\geq q$, $\xi\notin B_n$, and $F\cap C_n=\emptyset$.  If $\alpha\in C_n\cap\supp_L(\xi)$, then $\alpha\notin F$, and hence
\[
        p_\alpha(\xi)<h(\alpha)+q\leq h(\alpha)+n<f_n(\alpha).
\]
Thus $(f_n)_{n<\omega}$ satisfies the tail-diagonalization condition at $\xi$.  Proposition~\ref{prop:tail-equivalence} gives $X_L\in\Delta$.
\end{proof}

\begin{corollary}\label{cor:non-delta-unbounded-positions}
If $X_L\notin\Delta$, then for every $h:S\to\omega$ there is $\xi\in D$ such that, for every $q<\omega$, the set
\[
        \{\alpha\in\supp_L(\xi):p_\alpha(\xi)\geq h(\alpha)+q\}
\]
is infinite.
\end{corollary}

\begin{proof}
This is the contrapositive of Proposition~\ref{prop:position-bounded}.
\end{proof}

The positive conditions form the chain
\[
\begin{gathered}
\text{finite support}\Longrightarrow
\text{finite tail multiplicity}\Longrightarrow
\text{eventual position-bounded tails}\Longrightarrow X_L\in\Delta.
\end{gathered}
\]
For the first implication, take $h\equiv0$.  For the second, take $q_\xi=0$ in the definition of eventual position-boundedness.  The last implication is Proposition~\ref{prop:position-bounded}.  No converse implication is asserted.

A common-first-point construction gives a stationary-carrier example with
eventual position-boundedness and $\sigma$-closed discreteness.  It shows that the
equivalent conditions in Corollaries~\ref{cor:trichotomy}
and~\ref{cor:sigma-ideal} are not necessary for the $\Delta$-property.  Apart
from the added common first point, the tails have the shape used in
\cite[Example~5.2]{LS}.

\begin{definition}\label{def:common-first-point-example}
Fix a stationary set
\[
        S\subseteq E^{\omegaone}_\omega\cap\operatorname{acc}(\mathrm{Lim}(\omegaone))
\]
whose members are all above $\omega+1$.  Here $\mathrm{Lim}(\omegaone)$
denotes the set of limit ordinals below $\omegaone$.  The set
$\operatorname{acc}(\mathrm{Lim}(\omegaone))$ is club in $\omegaone$.
Deleting a bounded initial segment preserves stationarity, so such an $S$
exists.  For each
$\alpha\in S$ fix a strictly increasing sequence
$\langle\gamma_\alpha(n):n<\omega\rangle$ of limit ordinals cofinal in
$\alpha$ with $\gamma_\alpha(0)>1$, and define
$L_\alpha=\{\lambda_\alpha(n):n<\omega\}$ by
\[
        \lambda_\alpha(0)=1,\qquad
        \lambda_\alpha(n+1)=\gamma_\alpha(n)+(n+1).
\]
Let $L=\{L_\alpha:\alpha\in S\}$.
\end{definition}

Removing the common first point $1$ leaves tails of the form considered in
\cite[Example~5.2]{LS}.  The system in that example is $\sigma$-closed discrete and hence a
$\Delta$-space.  By Corollary~\ref{cor:tail-thinning-same-topology}, this
removal does not change the associated topology.

\begin{proposition}\label{prop:common-first-point-example}
The ladder system $L$ of Definition~\ref{def:common-first-point-example} has the following properties.
\begin{enumerate}[label=(\roman*),leftmargin=2em]
\item $L$ is a ladder system.  Each $L_\alpha$ is strictly increasing, cofinal
in $\alpha$, and contained in $D$.
\item For every $\xi\in D$ and $\alpha\in\supp_L(\xi)$, $p_\alpha(\xi)\le k_\xi$,
where $k_\xi$ is the finite part in the representation
$\xi=\mu+k_\xi$, with $\mu$ a limit ordinal or $0$.  Hence $L$ has eventual position-bounded tails, $X_L\in\Delta$,
and $X_L$ is $\sigma$-closed discrete.
\item $\supp_L(1)=S$, so $W_A(1)=A$ for every $A\subseteq S$, and the club-small
subsets of $S$ are exactly the countable subsets.  In particular,
$\mathcal J_{\rm cb}=[S]^{\le\omega}$ and $S\notin\mathcal J_{\rm cb}$.
\end{enumerate}
\end{proposition}

\begin{proof}
(i) The sequence $(\lambda_\alpha(n))_{n<\omega}$ is strictly increasing.
$\lambda_\alpha(1)=\gamma_\alpha(0)+1>1$, and
$\lambda_\alpha(n+2)>\lambda_\alpha(n+1)$ because
$\gamma_\alpha(n+1)>\gamma_\alpha(n)$.  It is cofinal in $\alpha$, since
$\lambda_\alpha(n+1)\ge\gamma_\alpha(n)$.  Every term is a successor ordinal,
hence has cofinality different from $\omega$ and lies in
$D=\omegaone\setminus S$.

(ii) If $\xi=\lambda_\alpha(0)=1$ then $k_\xi=1$ and $p_\alpha(\xi)=0$.  If
$\xi=\lambda_\alpha(n+1)=\gamma_\alpha(n)+(n+1)$ with $\gamma_\alpha(n)$ a limit,
then uniqueness of the representation $\xi=\mu+k_\xi$ gives
$\mu=\gamma_\alpha(n)$ and $k_\xi=n+1=p_\alpha(\xi)$.  Either way
$p_\alpha(\xi)\le k_\xi$.  Taking $h\equiv0$ and $q_\xi=k_\xi+1$, the set
$\{\alpha\in\supp_L(\xi):p_\alpha(\xi)\ge h(\alpha)+q_\xi\}$ is empty, so $L$ has
eventual position-bounded tails and $X_L\in\Delta$ by
Proposition~\ref{prop:position-bounded}.  For $\sigma$-closed discreteness,
define $c:D\to\omega$ by sending each ordinal $\xi$ to its finite part
$k_\xi$ in the representation $\xi=\mu+k_\xi$ with $\mu$ a limit ordinal or
$0$.  Along $L_\alpha$, the value $1$ occurs at most twice, and each value
$k\ge2$ occurs at most once.  Hence $c\restriction L_\alpha$ is finite-to-one
for every $\alpha\in S$.  Lemma~\ref{lem:sigma-cd-coloring} shows that $X_L$ is $\sigma$-closed
discrete.

(iii) Since $\lambda_\alpha(0)=1$ for all $\alpha\in S$, $\supp_L(1)=S$ and
$W_A(1)=A$.  By Corollary~\ref{cor:trichotomy}, if $A$ is club-small, then
every fiber is countable.  In particular, $W_A(1)=A$ is countable.  The converse
follows from Lemma~\ref{lem:bounded-club-small}, since countable
subsets of $\omegaone$ are bounded.
Hence the club-small subsets of $S$ are the countable ones, and by
Corollary~\ref{cor:sigma-ideal} $\mathcal J_{\rm cb}=[S]^{\le\omega}$.
\end{proof}

The space associated with the ladder system in
Definition~\ref{def:common-first-point-example} is a $\Delta$-space with an
isolated point whose support is all of $S$.  In particular, $S$ is not
club-small.  The example shows that the equivalent conditions in
Corollaries~\ref{cor:trichotomy} and~\ref{cor:sigma-ideal}, namely
club-smallness, membership in $\mathcal J_{\rm cb}$, countability of every
support fiber, and nonstationarity of $\Tr_L(S)$, are not necessary for
$X_L\in\Delta$.  These conditions concern support-fiber cardinalities and the associated trace-accumulation set,
whereas the sufficient tail conditions in this subsection impose explicit
bounds on the positions $p_\alpha(\xi)$.  Lemma~\ref{lem:fodor-collapse}
shows that $\Tr_L(A)$ is never both stationary and co-stationary.  It is
either nonstationary or contains a club.

\begin{proposition}\label{prop:restriction-obstruction}
Suppose that $X_L\notin\Delta$.  Then there is an uncountable set $T\subseteq S$ such that the subspace
\[
        X_T=T\cup \bigcup_{\alpha\in T}L_\alpha
\]
is not a $\Delta$-space.  Moreover, $T$ may be chosen as the union of the active carrier sets appearing in a single non-diagonalizable evanescent ladder pair.
\end{proposition}

\begin{proof}
Since $X_L\notin\Delta$, Proposition~\ref{prop:tail-equivalence} gives an evanescent ladder pair $(B_n,C_n)_{n<\omega}$ with no tail diagonalization.  Put
\[
        T=\bigcup_{n<\omega}C_n,
        \qquad
        D_T=\bigcup_{\alpha\in T}L_\alpha.
\]
Consider the subspace $X_T=T\cup D_T$.  Its non-isolated points are the
members of $T$, and their inherited basic neighborhoods are the corresponding
tails of $L_\alpha$.  Thus it is the ladder-system subspace generated by $T$.

If $X_T$ were a $\Delta$-space, then the restricted pair
\[
        (B_n\cap D_T,C_n)_{n<\omega}
\]
would admit a tail diagonalization $(f_n)_{n<\omega}$ for the restricted system on $T$.  For each $\xi\in D_T$, the tail-diagonalization requirements for the original and restricted systems coincide at $\xi$.  A point
$\xi\in D\setminus D_T$ belongs to no ladder indexed by any $C_n$, and since
$\bigcap_n B_n=\emptyset$, there is $n$ with $\xi\notin B_n$.  Thus $(f_n)_{n<\omega}$ is also a tail diagonalization of the original pair for $L$, contradicting
the choice of the pair.  Therefore $X_T$ is not a $\Delta$-space.  If $T$ were countable, then $X_T$ would be a countable $T_1$ space and hence a $\Delta$-space by the argument in Section~\ref{sec:preliminaries}.  Hence $T$ is uncountable.
\end{proof}

\begin{corollary}\label{cor:localized-no-ftm}
If $X_L\notin\Delta$, then there is an uncountable $T\subseteq S$ such that
$L\restriction T$ does not have finite tail multiplicity.  The associated restricted space is
\[
        X_T=T\cup\bigcup_{\alpha\in T}L_\alpha.
\]
More explicitly, for every $h:T\to\omega$ there is a point
\[
        \xi\in \bigcup_{\alpha\in T}L_\alpha
\]
such that
\[
        \{\alpha\in T:\xi\in L_\alpha\text{ and }p_\alpha(\xi)\geq h(\alpha)\}
\]
is infinite.
\end{corollary}

\begin{proof}
By Proposition~\ref{prop:restriction-obstruction}, there is an uncountable set $T\subseteq S$ such that $X_T$ is not a $\Delta$-space.  Apply Corollary~\ref{cor:non-delta-tail} to the restricted ladder system on $T$.  Its isolated part is
$D_T=\bigcup_{\alpha\in T}L_\alpha$.  The asserted infinitude condition is exactly the
failure of finite tail multiplicity for that restriction.
\end{proof}

Corollaries~\ref{cor:localized-no-ftm}
and~\ref{cor:non-delta-unbounded-positions} give two necessary conditions on
a ladder system whose associated space is not a $\Delta$-space.  An infinite support alone is insufficient.  After every finite thinning, some
isolated point must still lie on infinitely many remaining tails.  Moreover, for
every proposed height function, some isolated point must appear at positions
exceeding the proposed bounds by arbitrarily large amounts on infinitely many
ladders.

\subsection{Forcing absoluteness at the positive boundary}\label{sec:forcing-positive-boundary}
The positive tail conditions are upward absolute in the following limited sense.
This is not an absoluteness theorem for the $\Delta$-property itself.  It says
only that a ground-model tail witness remains valid when the same ladders are
used to form the space in an extension.

\begin{proposition}\label{prop:positive-boundary-absolute}
Let $V\subseteq W$ be transitive models of ZFC with
$\omegaone^{V}=\omegaone^{W}$.  Suppose that $L\in V$ is a ladder system on
$S\subseteq E^{\omegaone}_\omega$ in $V$.  In $W$, let $X_L^W$ denote the ladder-system space on the same underlying set generated by the ground-model family $L$.  If $L$ has finite tail multiplicity in $V$, then $X_L^W\in\Delta$.  If $L$ has eventual position-bounded tails in $V$, then it also has eventual position-bounded tails in $W$, and hence $X_L^W\in\Delta$.
\end{proposition}

\begin{proof}
Because $\omegaone^{V}=\omegaone^{W}$ and $S,L\in V$, the set
$D=\omegaone\setminus S$ and the ladder enumerations, supports, and positions
determined by $L$ are the same in $V$ and $W$.

Let $h:S\to\omega$ witness finite tail multiplicity in $V$.  For each $\xi\in D$, the set
\[
        F_\xi=\{\alpha\in\supp_L(\xi):p_\alpha(\xi)\geq h(\alpha)\}
\]
is finite in $V$.  Finiteness of a ground-model set of ordinals is absolute to $W$.  Hence the same $h$ witnesses finite tail multiplicity in $W$, and Proposition~\ref{prop:finite-tail}, applied in $W$, gives $X_L^W\in\Delta$.

Now let $h:S\to\omega$ witness eventual position-boundedness in $V$.  Fix $\xi\in D$.  In $V$ there is $q<\omega$ such that
\[
        F_{\xi,q}=\{\alpha\in\supp_L(\xi):p_\alpha(\xi)\geq h(\alpha)+q\}
\]
is finite.  This set is defined in $V$ from the ground-model ladder system and the
function $h$, and its finiteness is absolute to $W$.  Hence the same $h$ witnesses eventual position-boundedness in $W$, and Proposition~\ref{prop:position-bounded}, applied in $W$, gives $X_L^W\in\Delta$.
\end{proof}

Under the hypotheses of Proposition~\ref{prop:positive-boundary-absolute},
if the space generated in the extension by the ground-model ladder system $L$
is non-$\Delta$, then $L$ did not have eventual position-bounded tails in the
ground model.  In particular, it did not have finite tail multiplicity there.

\subsection{Carrier ideals for tail multiplicity}\label{sec:sigma-ftm}

\begin{definition}\label{def:sigma-ftm}
For $A\subseteq S$ and $h:S\to\omega$, put
\[
 W^h_A(\xi)=\{\alpha\in A\cap\supp_L(\xi):p_\alpha(\xi)\ge h(\alpha)\}
 \qquad(\xi\in D).
\]
Write $A\in\Jfin$ if there is $h:S\to\omega$ such that every $W^h_A(\xi)$ is finite.  Write $A\in\mathcal J^\ast$ if there is $h:S\to\omega$ such that every $W^h_A(\xi)$ is countable.  Thus
$S\in\Jfin$ means that $L$ has finite tail multiplicity (FTM), whereas
$S\in\mathcal J^\ast$ means that $L$ is $\sigma$-FTM.
\end{definition}

Both families admit an exact description.  On nonstationary carriers,
the proof constructs pairwise disjoint tails gap by gap.  Fodor's lemma gives
the stationary obstruction.

\begin{proposition}\label{prop:tail-multiplicity-ideals}
For every ladder system $L$ on $S\subseteq E^{\omegaone}_\omega$,
\[
        \Jfin=\mathcal J^\ast=\mathrm{NS}\restriction S.
\]
More precisely, fix $A\subseteq S$.
\begin{enumerate}[label=\textup{(\roman*)},leftmargin=2.4em]
\item If $A$ is nonstationary, then the ladders indexed by $A$ have
      pairwise disjoint tails.
\item If $A$ is stationary, then for every $h:S\to\omega$ there exist
      $\xi\in D$ and a stationary $A'\subseteq A$ such that
      $p_\alpha(\xi)=h(\alpha)$ for every $\alpha\in A'$.
\end{enumerate}
Consequently, when $S$ is stationary, both ideals are normal $\sigma$-ideals on $S$.
Moreover, if $S$ is stationary, then
\[
        \mathcal J_{\rm cb}\subsetneq\Jfin=\mathcal J^\ast.
\]
\end{proposition}

\begin{proof}
For (ii), fix $h$ and apply Fodor's lemma to the regressive map
$\alpha\mapsto\lambda_\alpha(h(\alpha))$ on $A$.  There are
$\xi\in D$ and a stationary $A'\subseteq A$ such that
$\lambda_\alpha(h(\alpha))=\xi$ for every $\alpha\in A'$.  Since each
ladder enumeration is strictly increasing, this equality implies
$p_\alpha(\xi)=h(\alpha)$, and hence $A'\subseteq W^h_A(\xi)$.  Thus no
stationary $A$ belongs to $\mathcal J^\ast$.

For (i), choose a club
$C=\{\gamma_\eta:\eta<\omegaone\}$, enumerated increasingly and
continuously, with $0\in C$ and $A\cap C=\emptyset$.  Each set
$A_\eta=A\cap(\gamma_\eta,\gamma_{\eta+1})$ is countable.  Enumerate it as
$\{\alpha_{\eta,n}:n<k_\eta\}$, where $k_\eta\le\omega$.  Within each gap, recursion along the enumeration deletes a finite initial segment from each
$L_{\alpha_{\eta,n}}$ so that the remaining tail lies above $\gamma_\eta$ and
misses the finitely many tails
chosen earlier in the same gap.  This is possible because distinct ordinal
ladders have finite intersection.  Tails from different gaps are disjoint,
so the resulting family is pairwise disjoint.  For $\alpha\in A$, let
$h_A(\alpha)$ be the length of the deleted initial segment, and extend
$h_A$ arbitrarily to a function $h:S\to\omega$.  Then every
$W_A^h(\xi)$ has size at most one.  Hence
$A\in\Jfin\subseteq\mathcal J^\ast$.

The ideal identities and the normality assertions follow because both ideals equal
$\mathrm{NS}\restriction S$.  The inclusion
$\mathcal J_{\rm cb}\subseteq\Jfin$ follows from
Corollary~\ref{cor:sigma-ideal} together with the stationary obstruction in part~(ii), applied to $h\equiv0$.  If $S$ is stationary, Fodor's lemma gives
$\xi_0\in D$ with stationary support $W=\supp_L(\xi_0)$.  An uncountable nonstationary subset of $W$ is obtained by recursively constructing an increasing
continuous sequence $\langle\gamma_\eta:\eta<\omegaone\rangle$ and points
$\alpha_\eta\in W$ such that
$\gamma_\eta<\alpha_\eta<\gamma_{\eta+1}$.  Then
$A=\{\alpha_\eta:\eta<\omegaone\}$ is uncountable and is disjoint from
the club $\{\gamma_\lambda:\lambda<\omegaone\text{ is a limit ordinal}\}$.
Thus $A$ is nonstationary, so $A\in\Jfin$, while
$W^0_A(\xi_0)=A$ is uncountable.  Therefore
$A\notin\mathcal J_{\rm cb}$.
\end{proof}

These ideal equalities can also be expressed in terms of tail thinnings.

\begin{definition}\label{def:point-countable}
A ladder system $L$ on $S$ is \emph{point-countable} if every support fiber
$\supp_L(\xi)$ is countable.
\end{definition}

\begin{corollary}\label{cor:nonstationary-tail-equivalences}
For a ladder system $L$ on $S\subseteq E^{\omegaone}_\omega$, the following
conditions are equivalent.
\begin{enumerate}[label=\textup{(\arabic*)},leftmargin=2.4em]
\item $S$ is nonstationary.
\item The system is FTM.
\item The system is $\sigma$-FTM.
\item Some tail thinning is pairwise disjoint.
\item Some tail thinning is point-countable.
\end{enumerate}
Under these conditions, $X_L$ is $\sigma$-closed discrete and hence a
$\Delta$-space.
\end{corollary}

\begin{proof}
Proposition~\ref{prop:tail-multiplicity-ideals} gives the equivalence of
(1)--(3) and the implication (1)$\Rightarrow$(4).  Conversely, (4) implies
(2), since pairwise disjoint tails have finite multiplicity.  Condition (3) says exactly that some tail thinning is point-countable, so it
is equivalent to (5).  If the thinned ladders are pairwise disjoint, assign to each point of a
thinned ladder its position on that ladder and assign arbitrary colors to the remaining isolated
points.  The resulting map is one-to-one on each thinned ladder
and finite-to-one on each original ladder.  Lemma~\ref{lem:sigma-cd-coloring} then shows that $X_L$ is
$\sigma$-closed discrete.  The permanence of the $\Delta$-property under
countable closed unions \cite[Prop.~2.2]{KLbasic} gives $X_L\in\Delta$.
\end{proof}

\begin{corollary}\label{cor:tail-multiplicity-hierarchy}
For every ladder system $L$,
\[
        \mathrm{FTM}\Longleftrightarrow\sigma\text{-FTM}
        \Longrightarrow\mathrm{EPB}\Longrightarrow X_L\in\Delta.
\]
If the carrier is stationary, then neither FTM nor $\sigma$-FTM holds.
Moreover, neither
$\mathrm{EPB}\Rightarrow\mathrm{FTM}$ nor
$X_L\in\Delta\Rightarrow\mathrm{FTM}$ is valid, even on stationary carriers.
\end{corollary}

\begin{proof}
The equivalence and the stationary-carrier obstruction follow from
Proposition~\ref{prop:tail-multiplicity-ideals}.  Finite tail multiplicity
implies EPB by taking $q_\xi=0$ for every $\xi$, and
Proposition~\ref{prop:position-bounded} gives $\mathrm{EPB}\Rightarrow X_L\in\Delta$.  The
common-first-point system of Proposition~\ref{prop:common-first-point-example}
is EPB and $\sigma$-closed discrete on a stationary carrier, hence it is a
$\Delta$-space, but Proposition~\ref{prop:tail-multiplicity-ideals} shows that
it is not FTM.
\end{proof}

\subsubsection{Tail multiplicity in the uniformization hierarchy}

Balogh, Eisworth, Gruenhage, Pavlov, and Szeptycki organize ladder systems into
a hierarchy of uniformization properties \cite[Figure~1]{BEGPS}.  For ladder-system spaces,
the property $M_{<\omega}$ corresponds to countable metacompactness and, by the
criterion of Leiderman and Szeptycki, to the $\Delta$-property
\cite[Prop.~4.1]{LS}.  They also consider a dual hierarchy of anti-uniformization properties \cite[Figure~2]{BEGPS}.  For stationary carriers, the formulas
defining thinness and the related weaker condition coincide with conditions
$(G_1)$ and $(G_2)$ of BEGPS \cite[p.~191, conditions~$(G_1)$ and~$(G_2)$]{BEGPS}.
For arbitrary carriers, the same formulas define the two conditions.

\begin{definition}\label{def:thin}
A ladder system $L$ on $S$ is \emph{thin} if, for every $f:\omegaone\to\omega$,
the set $\{\alpha\in S:|f''L_\alpha|=\aleph_0\}$ is nonstationary.  Equivalently,
for every $f$, the set of carrier points $\alpha$ for which $f\restriction L_\alpha$ has infinite
range is nonstationary in $\omegaone$.
For comparison, the weaker condition requires, for every $f$, that the set
$\{\alpha\in S:f\restriction L_\alpha\text{ is finite-to-one}\}$ be
nonstationary.  This is condition $(G_2)$, and thinness implies it.
\end{definition}

A problem raised in \cite[p.~192]{BEGPS} is whether it is consistent that a
ladder system on a stationary subset of $\omegaone$ is simultaneously thin and
countably metacompact.  Related results on existence, characterization, and nonexistence appear in
\cite[Theorems~10, 17, 23, and~24]{BEGPS}.

On stationary carriers, $\sigma$-closed discreteness and thinness are
incompatible.

\begin{proposition}\label{prop:sigma-cd-not-thin}
Let $L$ be a ladder system on a stationary
$S\subseteq E^{\omegaone}_\omega$.  If $X_L$ is $\sigma$-closed discrete, then
$L$ is not thin.  Hence $\sigma$-closed discreteness and thinness are mutually
exclusive on stationary carriers.  Moreover, every $\sigma$-closed-discrete
ladder-system space on a stationary carrier is a non-thin $\Delta$-space.
\end{proposition}

\begin{proof}
Let $c:D\to\omega$ be finite-to-one on every $L_\alpha$ as in
Lemma~\ref{lem:sigma-cd-coloring}, and extend $c$ arbitrarily to
$f:\omegaone\to\omega$.  Then
$\{\alpha:f\restriction L_\alpha\text{ is finite-to-one}\}=S$, contradicting
condition $(G_2)$ in Definition~\ref{def:thin}.  Thus $L$ is not thin.  The conclusion that every $\sigma$-closed-discrete ladder-system space on a stationary carrier is a non-thin $\Delta$-space follows from the closed-union permanence of $\Delta$-spaces
\cite[Prop.~2.2]{KLbasic}.
\end{proof}

\begin{corollary}\label{cor:sigma-ftm-thin}
Every $\sigma$-FTM ladder system is thin, and its associated space is
$\sigma$-closed discrete and hence a $\Delta$-space.  In particular, every
point-countable ladder system has these properties.  Such systems have nonstationary carriers and therefore do not provide
stationary-carrier instances of the thin/countably metacompact problem of
\cite[p.~192]{BEGPS}.
\end{corollary}

\begin{proof}
By Proposition~\ref{prop:tail-multiplicity-ideals}, a $\sigma$-FTM system has
nonstationary carrier.  Every subset of a nonstationary carrier is
nonstationary, so Definition~\ref{def:thin} gives thinness.  Corollary~\ref{cor:nonstationary-tail-equivalences}
gives $\sigma$-closed discreteness and the $\Delta$-property.  A point-countable
system is $\sigma$-FTM with zero truncation.
\end{proof}

In the Suslin-tree extension, Corral and Szeptycki use a countable
ground-model family of total systems.  Their argument first extends a system on
a stationary carrier to a total system.  Theorem~3.1 then produces a
single function that is eventually one-to-one on every ladder in the resulting
family.  Restricting this function to the original ladders gives the stated
conclusion for the original ladder system.

\begin{proposition}[{\cite[pp.~43--44, discussion preceding Theorem~3.1 and Theorem~3.1]{CS}}]\label{prop:cs-positive-consistency}
In the extension obtained by forcing with the Suslin tree over a model of
$\mathrm{MA}_{\omegaone}(\mathcal K)$, for every ladder system $L$ on a
stationary carrier $S$ there is a map $f:\omegaone\to\omega$ such that
$f\restriction L_\alpha$ is eventually one-to-one for every $\alpha\in S$.
\end{proposition}

Restricting $f$ to $D$ and applying Lemma~\ref{lem:sigma-cd-coloring} shows
that $X_L$ is $\sigma$-closed discrete.

In the same model, Corral and Szeptycki prove $M_{<\omega}$ for total
ladder systems \cite[Theorem~2.3]{CS}.  Extending a ladder system on a stationary carrier to a total system
and then restricting the resulting witness to the original ladders gives
$M_{<\omega}$ and therefore $X_L\in\Delta$.
Proposition~\ref{prop:cs-positive-consistency} additionally gives the
$\sigma$-closed discreteness of $X_L$.

\subsection{Tail traces and remainders}\label{sec:trace-remainder}

The isolated-part ideal gives rise to a generalized Boolean algebra of remainder
traces, all disjoint from the carrier part of the remainder.

For the remainder analysis, consider the ideal
\[
        \mathcal I_L=\{A\subseteq D:|A\cap L_\alpha|<\omega\text{ for every }\alpha\in S\}.
\]
Membership in $\mathcal I_L$ is defined solely by the finiteness of the
intersections with individual ladders.  By contrast, the tail-diagonalization
criterion in Proposition~\ref{prop:tail-equivalence} explicitly compares the
positions of points on the active ladders with integer thresholds.

For $C\subseteq S$ and $g:C\to\omega$, define the tail trace
\[
        T(C,g)=\{\xi\in D:\text{there is }\alpha\in C\cap\supp_L(\xi)
        \text{ with }p_\alpha(\xi)\geq g(\alpha)\}.
\]
For $B\subseteq D$, the set $B\cup C\cup T(C,g)$ is open in $X_L$.  Its isolated part is
\[
        B\cup T(C,g).
\]
In this notation, Proposition~\ref{prop:tail-equivalence} says that
$X_L\in\Delta$ if and only if every evanescent pair $(B_n,C_n)$ admits
functions $g_n:C_n\to\omega$ such that
\[
        \bigcap_{n<\omega}\bigl(B_n\cup T(C_n,g_n)\bigr)=\emptyset.
\]

This restates the tail-diagonalization criterion in tail-trace notation.  It is not a remainder characterization, because decreasing carriers
and position thresholds still appear explicitly.  The compactification results instead
describe the clopen information carried by subsets of the isolated part and by
their remainder traces.

\subsubsection{Closed traces in the Stone--\v{C}ech remainder}\label{sec:closed-traces}

The ideal $\mathcal I_L$ has a direct topological meaning.  Its members are
exactly the clopen subsets of $X_L$ contained in the isolated part $D$.  Their
remainder traces form a clopen trace algebra in the remainder.

\begin{proposition}\label{prop:ideal-closed-discrete}
For $A\subseteq D$, the following are equivalent.
\begin{enumerate}[label=(\roman*),leftmargin=2em]
\item $A\in\mathcal I_L$.
\item $A$ is closed in $X_L$.
\item $A$ is closed and discrete in $X_L$.
\item $A$ is clopen in $X_L$.
\end{enumerate}
\end{proposition}

\begin{proof}
Since each point of $D$ is isolated, every subset of $D$ is open and discrete.  Hence it remains to determine which subsets of $D$ are closed.

Suppose first that $A\in\mathcal I_L$.  Let $\alpha\in S$.  Then
$A\cap L_\alpha$ is finite, so some tail $N(\alpha,m)$ misses $A$.  Thus no
point of $S$ lies in $\cl_{X_L}A$.  Since points of $D$ are isolated,
$\cl_{X_L}A=A$.  Hence $A$ is closed.  As $A\subseteq D$ is open and discrete,
it is clopen.

Conversely, suppose that $A\subseteq D$ is closed in $X_L$.  If
$A\cap L_\alpha$ were infinite for some $\alpha\in S$, then every neighborhood of
$\alpha$ would meet $A$, so $\alpha\in\cl_{X_L}A$.  Since $A$ is closed, this would give
$\alpha\in A\subseteq D$, contradicting $\alpha\in S$.  Hence $A\cap L_\alpha$ is finite for
every $\alpha\in S$, so $A\in\mathcal I_L$.
\end{proof}

Every infinite $A\in\mathcal I_L$ therefore determines a subspace of
the remainder homeomorphic to $\beta A\setminus A$.

\begin{corollary}\label{cor:ideal-remainder-subspace}
If $A\in\mathcal I_L$ is infinite, then
\[
        \cl_{\beta X_L}A\cong\beta A,
        \qquad
        \cl_{\beta X_L}A\setminus A\subseteq\beta X_L\setminus X_L.
\]
Hence the remainder of $X_L$ contains a subspace homeomorphic to
$\beta A\setminus A$.
\end{corollary}

\begin{proof}
By Proposition~\ref{prop:ideal-closed-discrete}, the set $A$ is clopen and
therefore $C^*$-embedded in $X_L$.  By \cite[Theorem~6.9(a)]{GJ}, $\cl_{\beta X_L}A\cong\beta A$.  Because $A$ is closed in $X_L$, every point of $\cl_{\beta X_L}A\setminus A$ lies outside $X_L$ and hence in $\beta X_L\setminus X_L$.
\end{proof}

\begin{proposition}\label{prop:large-remainder}
If the carrier $S$ is uncountable, then $X_L$ is neither Lindel\"of nor
pseudocompact, and the remainder $\beta X_L\setminus X_L$ contains a subspace
homeomorphic to $\omega^*=\beta\omega\setminus\omega$.
\end{proposition}

\begin{proof}
The carrier $S$ is an uncountable closed discrete subspace of $X_L$, so
$X_L$ is not Lindel\"of.  It remains to construct a countably infinite member
of $\mathcal I_L$.  Fix a strictly increasing sequence
$\langle\alpha_n:n<\omega\rangle$ in $S$ with supremum $\delta$, and choose
points $\xi_n\in L_{\alpha_n}$ recursively so that
$\xi_{n+1}>\alpha_n$.  If $\delta\in S$,
require in addition that $\xi_n\notin L_\delta$.  This is possible because each
$L_{\alpha_n}$ is cofinal in $\alpha_n$ and
$L_{\alpha_n}\cap L_\delta$ is finite.  Put
$A=\{\xi_n:n<\omega\}$.  If $\delta\in S$, then
$A\cap L_\delta=\emptyset$.  If $\beta>\delta$, then $A\subseteq\delta$ and
$L_\beta\cap\delta$ is finite.  If $\beta<\delta$, then
$A\cap L_\beta\subseteq A\cap\beta$ is finite because
$\xi_{n+1}>\alpha_n$ and $\sup_n\alpha_n=\delta$.  Hence
$A\in\mathcal I_L$.

By Proposition~\ref{prop:ideal-closed-discrete}, the set $A$ is clopen.  Let $u:A\to\mathbb{R}$ be unbounded.  Extend $u$ to $X_L$ by setting it equal to zero on
$X_L\setminus A$.  The resulting function is continuous and unbounded on $X_L$,
so $X_L$ is not pseudocompact.  Corollary~\ref{cor:ideal-remainder-subspace} gives a subspace
of the remainder homeomorphic to $\beta A\setminus A\cong\omega^*$.
\end{proof}

Write
\[
\begin{gathered}
        X_L^*=\beta X_L\setminus X_L,
        \qquad
        A^*=\cl_{\beta X_L}A\setminus X_L
        \quad(A\subseteq X_L),\\
        A\subseteq^*B
        \quad\Longleftrightarrow\quad
        |A\setminus B|<\omega.
\end{gathered}
\]
The clopen traces form a generalized Boolean algebra in the remainder.

\begin{proposition}\label{prop:trace-algebra-embedding}
For $A,B\in\mathcal I_L$,
\[
\begin{aligned}
 A^{*}\cap B^{*}&=(A\cap B)^{*},&
 A^{*}\cup B^{*}&=(A\cup B)^{*},\\
 A^{*}\setminus B^{*}&=(A\setminus B)^{*}.
\end{aligned}
\]
Also,
\[
        A^{*}\subseteq B^{*}
        \quad\Longleftrightarrow\quad
        A\subseteq^{*}B.
\]
Consequently, $A\mapsto A^{*}$ induces an isomorphism of generalized Boolean
algebras from $\mathcal I_L/\mathrm{fin}$ onto its image in
$\operatorname{Clop}(X_L^{*})$.
\end{proposition}

\begin{proof}
By Proposition~\ref{prop:ideal-closed-discrete}, the sets $A$ and $B$ are
clopen in $X_L$.  Their characteristic functions extend continuously to
$\beta X_L$, and Boolean operations on the extensions give the three displayed
identities.  For a clopen discrete subset $C$ of $X_L$, $C^{*}=\emptyset$ if and only if
$C$ is finite.  Hence
\[
        A^{*}\subseteq B^{*}
        \Longleftrightarrow
        (A\setminus B)^{*}=\emptyset
        \Longleftrightarrow
        A\subseteq^{*}B.
\]
Together, these identities give the asserted isomorphism onto the image.
\end{proof}

The identities in Proposition~\ref{prop:trace-algebra-embedding} do not
involve the position thresholds $p_\alpha(\xi)$ from the $\Delta$-criteria.
Compact ladder tails, on the other hand, contribute nothing to the
Stone--\v{C}ech remainder.  For every $\alpha\in S$ and
$m<\omega$,
\[
        \cl_{\beta X_L}N(\alpha,m)=N(\alpha,m).
\]
The set $N(\alpha,m)$ is compact by Section~\ref{sec:spaces} and therefore closed in the Hausdorff space $\beta X_L$.  It follows that infinite
members of $\mathcal I_L$ have nonempty clopen remainder traces, whereas
individual compact ladder tails have none.

\subsubsection{The \texorpdfstring{$S$}{S}-part of the remainder}

Write
\[
        E_S=\cl_{\beta X_L}S\setminus S
\]
for the \emph{$S$-part} of the remainder.  Every member of the clopen trace algebra is disjoint
from $E_S$.  The carrier closure, however, retains compactification
information not seen by that algebra.  When $X_L$ is normal, $E_S$ is homeomorphic to
$\beta S\setminus S$.

\begin{proposition}\label{prop:trace-separated}
For every $A\in\mathcal I_L$, the sets $A$ and $S$ are completely separated in
$X_L$.  Hence $\cl_{\beta X_L}A\cap\cl_{\beta X_L}S=\emptyset$, and every trace
$A^{*}$ is disjoint from $E_S$.
\end{proposition}

\begin{proof}
By Proposition~\ref{prop:ideal-closed-discrete}, the set $A$ is clopen, so its
characteristic function $\chi_A:X_L\to\{0,1\}$ is continuous.  It is $1$ on $A$
and $0$ on $S$, so $A$ and $S$ are completely separated.  Completely separated
sets have disjoint closures in $\beta X_L$.  Hence
$A^{*}\subseteq\cl_{\beta X_L}A$ is disjoint from
$\cl_{\beta X_L}S\supseteq E_S$.
\end{proof}

\begin{proposition}\label{prop:delta-trace-invariant}
For every $A\in\mathcal I_L$, the set $X_L\setminus A$ is clopen and
\[
        X_L\in\Delta\quad\Longleftrightarrow\quad X_L\setminus A\in\Delta.
\]
\end{proposition}

\begin{proof}
Proposition~\ref{prop:ideal-closed-discrete} shows that $A$ is clopen, so
$X_L=A\oplus(X_L\setminus A)$.  The forward implication follows from subspace
permanence.  For the converse, $A$ is discrete and hence a $\Delta$-space, so
topological-sum permanence applies.
\end{proof}

By \cite[Theorem~6.9(a)]{GJ}, the inclusion $S\hookrightarrow X_L$
induces a homeomorphism
\[
        \beta S\longrightarrow\cl_{\beta X_L}S
\]
if and only if $S$ is $C^{*}$-embedded in $X_L$.

The connection between constant-function uniformization and normality of
ladder-system spaces is described as folklore in \cite[p.~189]{BEGPS}.  The
remainder analysis uses the equivalent formulation in terms of the
$C^{*}$-embedding of the carrier.

\begin{proposition}\label{prop:S-remainder-normal}
The subspace $S$ is $C^{*}$-embedded in $X_L$ if and only if $X_L$ is normal.
Consequently, the map $\beta S\to\cl_{\beta X_L}S$ induced by inclusion is a
homeomorphism if and only if $X_L$ is normal.  When $X_L$ is not normal,
$\cl_{\beta X_L}S$ is a proper compactification of the
discrete space $S$ below $\beta S$, and $E_S$ is its remainder rather than the
full Stone--\v{C}ech remainder $\beta S\setminus S$.
\end{proposition}

\begin{proof}
If $X_L$ is normal, then the closed set $S$ is $C^{*}$-embedded by the Tietze
extension theorem.

Conversely, suppose $S$ is $C^{*}$-embedded.  Let $H$ and $K$ be disjoint closed subsets of $X_L$.  It suffices to separate them by disjoint open sets.  Put $H_S=H\cap S$, $K_S=K\cap S$,
$H_D=H\cap D$, and $K_D=K\cap D$.  Define $f:S\to\{0,1\}$ by $f=1$ on $K_S$ and
$f=0$ elsewhere on $S$, and extend it to a continuous $g:X_L\to[0,1]$ by
$C^{*}$-embeddedness.  For each $\alpha\in S$,
$g(\lambda_\alpha(n))\to g(\alpha)=f(\alpha)\in\{0,1\}$.  For $\xi\in D$, define
$F(\xi)=1$ if $g(\xi)\ge\tfrac12$ and $F(\xi)=0$ otherwise.  Then $F\restriction L_\alpha$ is eventually constant with value $f(\alpha)$ for every $\alpha\in S$.  Set
\[
        V=\{\xi\in D:F(\xi)=1\}\cup\{\alpha\in S:f(\alpha)=1\}.
\]
The eventual constancy of $F\restriction L_\alpha$ with value $f(\alpha)$ on every ladder implies that $V$ is clopen.  Moreover $K_S\subseteq V$ and
$H_S\cap V=\emptyset$.

The sets $H_D,K_D\subseteq D$ consist of isolated points and are disjoint.  Define
\[
        U_K=(V\cup K_D)\setminus H_D,\qquad U_H=((X_L\setminus V)\cup H_D)\setminus K_D.
\]
The sets $U_K$ and $U_H$ contain $K$ and $H$, respectively, and are disjoint.
$K_S\subseteq V\subseteq U_K$ and $K_D\subseteq U_K$.  Dually,
$H_S\subseteq X_L\setminus V\subseteq U_H$ and $H_D\subseteq U_H$.  Every point of $K_D$ is excluded from $U_H$, and every point of $H_D$ is
excluded from $U_K$.  Any remaining common point would have to lie in both $V$ and
$X_L\setminus V$.

It remains to see that $U_K$ and $U_H$ are open.  Every point of
$D$ is isolated, so it suffices to verify openness at the non-isolated points
of $U_K$ and $U_H$.  If $\alpha\in S\cap U_K$, then $\alpha\in V$.  Since $V$ is clopen, there is $m_0<\omega$ with
$N(\alpha,m_0)\subseteq V$.  Because $H_S\cap V=\emptyset$, $\alpha\notin H_S$.  Since $H$ is closed and $H_D\subseteq H$,
$\alpha\notin\cl_{X_L}H_D$, so there is $m_1<\omega$ with
$N(\alpha,m_1)\cap H_D=\emptyset$.  For
$m=\max\{m_0,m_1\}$,
\[
        N(\alpha,m)\subseteq V\setminus H_D\subseteq U_K.
\]
Symmetrically, if $\alpha\in S\cap U_H$, then
$\alpha\in X_L\setminus V$.  Since $X_L\setminus V$ is clopen, there is $m_0<\omega$ with
$N(\alpha,m_0)\subseteq X_L\setminus V$.  Because $K_S\subseteq V$, $\alpha\notin K_S$.  Since $K$ is closed and $K_D\subseteq K$,
$\alpha\notin\cl_{X_L}K_D$, so there is $m_1<\omega$ with
$N(\alpha,m_1)\cap K_D=\emptyset$.  For
$m=\max\{m_0,m_1\}$,
\[
        N(\alpha,m)\subseteq (X_L\setminus V)\setminus K_D\subseteq U_H.
\]
Thus $U_K$ and $U_H$ are open, and $X_L$ is normal.

The remaining compactification statements follow from the equivalence between normality and $C^{*}$-embeddedness and from the universal
property of $\beta S$.  If $S$ is not $C^{*}$-embedded in $X_L$, the compactification
$\cl_{\beta X_L}S$ is strictly below $\beta S$.
\end{proof}

Proposition~\ref{prop:trace-algebra-embedding} describes the clopen trace
algebra.  Proposition~\ref{prop:trace-separated} shows that its members are disjoint
from the $S$-part $E_S$, while Proposition~\ref{prop:S-remainder-normal}
determines that part in the normal case.  If $X_L$ is normal, then
$E_S\cong\beta S\setminus S$.  Hence, among normal ladder-system spaces whose
carriers have the same cardinality, the homeomorphism type of $E_S$ is
independent of the particular ladder system $L$.  No remainder
characterization of the $\Delta$-property is asserted here.

\subsection{The ideal of \texorpdfstring{$\sigma$}{sigma}-closed-discrete carriers}\label{sec:cC}

Unlike FTM and $\sigma$-FTM, $\sigma$-closed discreteness can already hold on
a stationary carrier in ZFC, as
Proposition~\ref{prop:common-first-point-example} shows.  For a subcarrier
$A\subseteq S$, the restricted form of Lemma~\ref{lem:sigma-cd-coloring}
characterizes the $\sigma$-closed discreteness of $X_A$ in the sense of
Definition~\ref{def:sigma-cd}.  Write each $L_\alpha\subseteq D$ in increasing order as
$\{\lambda_\alpha(j):j<\omega\}$, and write $p_\alpha(\xi)=j$ when
$\xi=\lambda_\alpha(j)$.

\begin{definition}\label{def:cC}
A set $A\subseteq S$ is a \emph{$\sigma$-closed-discrete carrier},
written $A\in\cC_L$, if the restricted ladder-system space $X_A$ is $\sigma$-closed discrete.
Since the ladders indexed by $A$ lie in $D_A$, the restricted form of
Lemma~\ref{lem:sigma-cd-coloring} gives
\[
   A\in\cC_L\Longleftrightarrow
   \exists\,c:D_A\to\omega\ \text{such that }c\restriction L_\alpha
   \text{ is finite-to-one for every }\alpha\in A.
\]
Equivalently, a witness may be taken to be a map $c:D\to\omega$, because a witness on $D_A$
extends arbitrarily to $D$, while a global witness restricts to $D_A$.
The full space $X_L$ is $\sigma$-closed discrete if and only if $S\in\cC_L$.
The terms ``ideal'' and ``$\sigma$-ideal'' are used without a properness
requirement.  Properness is stated separately when relevant.
\end{definition}

In terms of the ladder enumerations, $A\in\cC_L$ if and only if there is
$c:D\to\omega$ such that $c(\lambda_\alpha(j))\to\infty$ as $j\to\infty$
for every $\alpha\in A$.  The restriction $c\restriction L_\alpha$ is finite-to-one if and only if, for each $K$, only finitely many
points of $L_\alpha$ have color at most $K$.  Hence $A\in\cC_L$ exactly
when a single map $c:D\to\omega$ tends to infinity along every ladder indexed
by $A$.

\begin{proposition}\label{prop:cC-ideal}
$\cC_L$ is downward closed and closed under finite unions, and
$[S]^{\le\aleph_0}\subseteq\cC_L$.  Moreover, $\cC_L=\mathcal P(S)$ if and only if $L$ is
$\sigma$-closed discrete, and $\cC_L$ is a proper ideal if and only if $L$ is not
$\sigma$-closed discrete.
\end{proposition}

\begin{proof}
A witness for $A\in\cC_L$ also witnesses $B\in\cC_L$ whenever $B\subseteq A$, so $\cC_L$ is downward closed.  For finite unions, let $c_A$ and $c_B$ witness $A\in\cC_L$ and $B\in\cC_L$, respectively.  Choose a
pairing function $\langle\cdot,\cdot\rangle:\omega^2\to\omega$ with
$\langle a,b\rangle\ge\max(a,b)$, and put
$c(\xi)=\langle c_A(\xi),c_B(\xi)\rangle$.  Fix $\alpha\in A$ and $K<\omega$.  If $\xi\in L_\alpha$ and $c(\xi)\le K$, then
$c_A(\xi)\le K$, and there are only finitely many such $\xi$.  The same argument, with $c_B$, applies to every $\alpha\in B$.  Thus $c$ witnesses $A\cup B\in\cC_L$.

For countable $A$, $D_A=\bigcup_{\alpha\in A}L_\alpha$ is countable.  Any
injection $c:D_A\to\omega$ is finite-to-one on every $L_\alpha$, so
$A\in\cC_L$.

Finally, any witness for $S$ also works for every $A\subseteq S$.  Hence
$\cC_L=\mathcal P(S)$ if and only if $S\in\cC_L$,
that is, if and only if $L$ is $\sigma$-closed discrete.  Thus $\cC_L$ is proper if and only if $L$ is not
$\sigma$-closed discrete.
\end{proof}

The ideal $\cC_L$ differs from both the isolated-part ideal $\mathcal I_L$ and the club-small decomposition ideal
$\mathcal J_{\rm cb}$.  The ideal $\mathcal I_L$ is defined on $D$ and
controls clopen remainder traces, whereas $\cC_L$ and
$\mathcal J_{\rm cb}$ are ideals on $S$ with different defining conditions.
For the system of Proposition~\ref{prop:common-first-point-example}, the
stationary carrier $S$ belongs to $\cC_L$ because the full space is
$\sigma$-closed discrete, but $S$ does not belong to
$\mathcal J_{\rm cb}$.

\begin{proposition}\label{prop:nonstationary-cC}
If $A\subseteq S$ is nonstationary, then $A\in\cC_L$.  Hence
\[
        \mathcal J^\ast=\mathrm{NS}\restriction S\subseteq\cC_L.
\]
\end{proposition}

\begin{proof}
Choose a club $C\subseteq\omegaone$ with $C\cap A=\emptyset$, and, if
necessary, add $0$ to $C$.  Enumerate $C$ increasingly and continuously as
$C=\{\gamma_\eta:\eta<\omegaone\}$.  For each $\eta<\omegaone$ choose an injection
\[
        e_\eta:D\cap[\gamma_\eta,\gamma_{\eta+1})\longrightarrow\omega.
\]
Such an injection exists because each interval below $\omegaone$ is countable.  Define
$c:D\to\omega$ by $c(\xi)=e_\eta(\xi)$ when
$\xi\in D\cap[\gamma_\eta,\gamma_{\eta+1})$.

Fix $\alpha\in A$, and let $\eta$ be the unique index such that
$\alpha\in[\gamma_\eta,\gamma_{\eta+1})$.  Since $\alpha\notin C$,
$\gamma_\eta<\alpha<\gamma_{\eta+1}$.  Choose $h(\alpha)<\omega$ with
$\lambda_\alpha(j)>\gamma_\eta$ for all $j\ge h(\alpha)$.  Then the tail
$\{\lambda_\alpha(j):j\ge h(\alpha)\}$ lies in
$D\cap[\gamma_\eta,\gamma_{\eta+1})$, where $c=e_\eta$ is injective.  Because only finitely many initial points lie outside this tail,
$c\restriction L_\alpha$ is finite-to-one.  Thus $c$ witnesses $A\in\cC_L$ by
Definition~\ref{def:cC}.  The inclusion $\mathcal J^\ast=\mathrm{NS}\restriction S\subseteq\cC_L$ then follows from Proposition~\ref{prop:tail-multiplicity-ideals}.
\end{proof}

\subsubsection{\texorpdfstring{$\sigma$}{sigma}-additivity and the threshold criterion}

Let $A=\bigcup_n A_n$, where $(A_n)$ is increasing and $A_n\in\cC_L$ for every $n$.  For each
$n$, choose a witness $c_n$ for $A_n$, and, for $\alpha\in A$, put
$n(\alpha)=\min\{n:\alpha\in A_n\}$.  Then
$c_\ell\restriction L_\alpha$ is finite-to-one whenever
$\ell\ge n(\alpha)$.  A threshold function will select the local
coloring used at each isolated point.

\begin{definition}\label{def:U}
For $g:S\to\omega$, the \emph{threshold condition} $U(g)$ asserts that there is
$m:D\to\omega$ such that $\{\xi\in L_\alpha:m(\xi)<g(\alpha)\}$ is finite
for every $\alpha\in S$.
\end{definition}

Fact~4.2 of Carvalho, Inamdar, and Rinot concerns ladder systems on
stationary carriers.  In that setting, $(\forall g:S\to\omega)\,U(g)$ is the
strict-threshold form of \cite[Fact~4.2(3)]{CIR}.  A subset of a strictly increasing cofinal $\omega$-sequence is bounded below the supremum of the sequence if and only if it is finite.  Replacing $g$ by $g+1$ converts the weak inequality $\leq$ into the strict inequality $<$.  For arbitrary carriers, the equivalence follows directly from
Definition~\ref{def:Mlessomega}.  If $F$ witnesses $M_{<\omega}$ for $g$,
then $m(\xi)=1+\max(F(\xi)\cup\{0\})$ witnesses $U(g)$, because
$g(\alpha)\in F(\xi)$ for all but finitely many $\xi\in L_\alpha$, and at
each such point $m(\xi)>g(\alpha)$.  Conversely, if
$m$ witnesses $U(g+1)$, then
$F(\xi)=\{k<\omega:k<m(\xi)\}$ witnesses $M_{<\omega}$ for $g$, since outside
the finite exceptional set $m(\xi)\geq g(\alpha)+1$.  Thus the universal
threshold scheme is equivalent to $M_{<\omega}$.  Lemma~\ref{lem:Mlessomega-equiv}
and Proposition~\ref{prop:active-normal-form} then give
\[
       (\forall g:S\to\omega)\,U(g)
       \quad\Longleftrightarrow\quad X_L\in\Delta.
\]
The gluing argument for $\cC_L$ uses only the single instance
$U(n(\cdot))$.

\begin{proposition}\label{prop:cC-sigma-delta}
If $X_L\in\Delta$, then $\cC_L$ is a $\sigma$-ideal.
\end{proposition}

\begin{proof}
Let $A=\bigcup_n A_n$ with $A_n\in\cC_L$.  Finite-union closure allows each $A_n$ to be replaced by
$\bigcup_{j\le n}A_j$.  The resulting sequence is increasing and has union $A$.  Put
$n(\alpha)=\min\{n:\alpha\in A_n\}$ for $\alpha\in A$.  For each $n$, choose a witness $c_n:D_{A_n}\to\omega$ and extend it
arbitrarily to $D_A$.  Then $c_n\restriction L_\alpha$ is finite-to-one for
every $\alpha\in A_n$.  Extend the map $n(\cdot):A\to\omega$ arbitrarily to $S$.  The threshold
criterion for $X_L\in\Delta$ then provides a map $m:D\to\omega$ such that
\[
        \{\xi\in L_\alpha:m(\xi)<n(\alpha)\}
\]
is finite for every $\alpha\in A$.

Choose a pairing function $\langle\cdot,\cdot\rangle:\omega^2\to\omega$ with $\langle r,s\rangle\ge\max\{r,s\}$.  For $\xi\in D_A$, set
\[
        c(\xi)=\langle m(\xi),c_{m(\xi)}(\xi)\rangle.
\]
Fix $\alpha\in A$, put $n=n(\alpha)$, and fix $K<\omega$.  Then
\[
\{\xi\in L_\alpha:c(\xi)\le K\}
\subseteq
\{\xi\in L_\alpha:m(\xi)<n\}
\cup
\bigcup_{\ell=n}^{K}
\{\xi\in L_\alpha:m(\xi)=\ell,
                   c_\ell(\xi)\le K\}.
\]
The first set is finite by the threshold condition.  For $n\le\ell\le K$, $\alpha\in A_\ell$.  Since
$c_\ell\restriction L_\alpha$ is finite-to-one, the set
$\{\xi\in L_\alpha:m(\xi)=\ell,\ c_\ell(\xi)\le K\}$ is finite.  Thus
$c\restriction L_\alpha$ is finite-to-one for every $\alpha\in A$, and
$A\in\cC_L$.
\end{proof}

If $L$ is $\sigma$-closed discrete, then $U(g)$ holds for every $g$.  A
witnessing coloring $c$ satisfies $c(\lambda_\alpha(j))\to\infty$, so taking
$m=c$ makes
$\{\xi\in L_\alpha:c(\xi)<g(\alpha)\}$ finite for every $\alpha$.  The universal threshold equivalence used in
Proposition~\ref{prop:cC-sigma-delta} then gives another proof that
$\sigma$-closed discreteness implies the $\Delta$-property.

\subsubsection{A consistency boundary}

\begin{corollary}\label{cor:cC-consistency-boundary}
In the Corral--Szeptycki model, $\cC_L$ is a $\sigma$-ideal for every ladder
system on every carrier $S\subseteq E^{\omegaone}_\omega$.
\end{corollary}

\begin{proof}
If $S$ is stationary, Proposition~\ref{prop:cs-positive-consistency} shows that
$L$ is $\sigma$-closed discrete, and Proposition~\ref{prop:cC-ideal} gives
$\cC_L=\mathcal P(S)$.  If $S$ is nonstationary, every subset of $S$ is
nonstationary, so Proposition~\ref{prop:nonstationary-cC} again gives
$\cC_L=\mathcal P(S)$.  In either case $\cC_L$ is a $\sigma$-ideal.
\end{proof}

\subsubsection{Intersection graphs and the remaining question}

Let $\GL$ be the \emph{intersection graph} of $L$.  Its vertex set is $S$.  For distinct $\alpha,\beta\in S$, set
$\alpha\sim\beta$ if and only if $L_\alpha\cap L_\beta\neq\emptyset$.

\begin{proposition}\label{prop:stationary-graph-clique}
If $S$ is stationary, then $\GL$ contains a clique of cardinality $\aleph_1$.
Consequently, $\chi(\GL)=\aleph_1$.
\end{proposition}

\begin{proof}
The regressive map $\alpha\mapsto\lambda_\alpha(0)$ is constant on a
stationary subset of $S$ by Fodor's lemma.  The ladders indexed by that set
share their first point and therefore form a clique of size $\aleph_1$.  Since the vertex set has cardinality $\aleph_1$,
$\chi(\GL)\le |S|=\aleph_1$.  Hence $\chi(\GL)=\aleph_1$.
\end{proof}

\begin{proposition}\label{prop:graph-cC-bridge}
Let $A\subseteq S$.
\begin{enumerate}[label=\textup{(\roman*)},leftmargin=2.4em]
\item If $A$ is an independent set of $\GL$, then $A\in\cC_L$.
\item If $\GL[A]$ has a countable proper coloring, then $A$ is a countable
      union of members of $\cC_L$.
\item If $\cC_L$ is not a $\sigma$-ideal, then there are pairwise disjoint
      sets $A_n\in\cC_L$ whose union is stationary and such that some $A_n$
      contains adjacent vertices of $\GL$.
\end{enumerate}
\end{proposition}

\begin{proof}
For (i), independence ensures that each point of $\bigcup_{\alpha\in A}L_\alpha$ belongs to a unique ladder indexed by $A$.  Define $c(\xi)=p_\alpha(\xi)$ for that unique $\alpha$, and put $c(\xi)=0$ off
$\bigcup_{\alpha\in A}L_\alpha$.  Then $c\restriction L_\alpha$ is
one-to-one for every $\alpha\in A$, so $A\in\cC_L$.

For (ii), every color class is independent and hence belongs to $\cC_L$ by
(i).  Thus $A$ is a countable union of members of $\cC_L$.

For (iii), choose $B_n\in\cC_L$ with
$B=\bigcup_n B_n\notin\cC_L$, and set
$A_n=B_n\setminus\bigcup_{j<n}B_j$.  By downward closure, the sets $A_n$ belong to $\cC_L$, are pairwise disjoint, and have union $B$.
Proposition~\ref{prop:nonstationary-cC} implies that $B$ is stationary.  If
every $A_n$ were independent, then assigning color $n$ to $A_n$ would be a
countable proper coloring of $\GL[B]$, contradicting
Proposition~\ref{prop:stationary-graph-clique} applied to the restricted
system on $B$.
\end{proof}

Known forcing and guessing constructions of non-$\Delta$ ladder-system spaces
show that the universal threshold scheme can fail.  See
\cite[Theorem~5.3]{LS} and \cite[Remark~4.3]{CIR}.  Such failures do not by
themselves produce a countable cover of $S$ by members of $\cC_L$, so they are
not counterexamples to the $\sigma$-additivity of $\cC_L$.  By
Proposition~\ref{prop:cC-sigma-delta}, any negative answer must be witnessed
by a non-$\Delta$ ladder-system space.

\begin{problem}\label{prob:cC-zfc}
Is the assertion ``for every carrier $S\subseteq E^{\omegaone}_\omega$ and every
ladder system $L$ on $S$, $\cC_L$ is a $\sigma$-ideal'' a theorem of
$\mathrm{ZFC}$?  Equivalently, for every
ladder system $L$ on $S$, if $S$ is a countable union of subcarriers $A_n$ such
that each restricted space $X_{A_n}$ is $\sigma$-closed discrete, is $X_L$ necessarily
$\sigma$-closed discrete?
\end{problem}

Determining which forcing assumptions make every ladder system on a
stationary carrier satisfy $M_{<\omega}$ is a separate problem.  The remainder
results here are structural separation statements, not a
compactification-theoretic characterization.

\section*{Data availability}
No research data were generated or analyzed in this study.

\section*{Acknowledgements}
The author thanks colleagues and friends for helpful discussions.

\end{document}